\documentclass[11pt]{article}
\pdfoutput=1

\usepackage[latin1]{inputenc}
\usepackage[english]{babel}
\usepackage{amsmath,amsfonts,amssymb}
\usepackage{euscript}
\usepackage{stmaryrd}
\usepackage{float}
\usepackage{url}
\usepackage{graphicx}
\usepackage{verbatim}
\usepackage{enumerate}
\usepackage{todonotes}
\usepackage{hyperref}
\usepackage{bbm}
\usepackage{cases}
\usepackage{eurosym}
\usepackage{algorithm}
\usepackage{algorithmic}
\usepackage[algo2e,boxruled]{algorithm2e}
\usepackage{comment}

\graphicspath{{./figures/}}

\usepackage[colors]{optsys}

\newcommand{\ajoute}[1]{\blue{#1}}

\newcommand{\timeset}[2]{\ic{#1,#2}}
\newcommand{\timesetbis}[2]{#1{:}#2}
\newcommand{\shortset}{\timeset{0}{M}}
\newcommand{\longset}{\timeset{0}{D}}
\newcommand{\longfinal}{D+1}
\newcommand{\price}{p}

\newcommand{\pricespace}{\YY}
\newcommand{\discone}{10}
\newcommand{\fhealth}{\mathfrak{N}}

\newcommand{\DP}{Dynamic Programming}
\newcommand{\SP}{Stochastic Programming}
\newcommand{\SDP}{Stochastic Dynamic Programming}
\newcommand{\SDDP}{Stochastic Dual Dynamic Programming}
\newcommand{\SDDIP}{Stochastic Dual Dynamic Integer Programming}

\renewcommand{\trib}{\tribu{F}}
\renewcommand{\filtration}{\trib_{\longset}}

\newcommand{\FractionOfTheCapacity}{\xi} 

\usepackage{a4wide}

\newcommand{\nb}[3]{
  {\colorbox{#2}{\bfseries\sffamily\tiny\textcolor{white}{#1}}}
  {\textcolor{#2}{\text{$\blacktriangleright$}
      {\textcolor{#2}{#3}}\text{$\blacktriangleleft$}}}}

\newcommand{\pierre}[1]{\nb{Pierre}{blue}{#1}}

\title{Decomposition Methods for Dynamically Monotone
       Two-Time-Scale Stochastic Optimization Problems}

\author{Tristan Rigaut\thanks{Schneider Electric, Grenoble, France},
        Pierre Carpentier\thanks{UMA, ENSTA Paris, IP Paris, Palaiseau, France},
        Jean-Philippe Chancelier\thanks{CERMICS, Ecole des Ponts, Marne-la-Vall\'{e}e, France},
        Michel De Lara\footnotemark[3]}

\date{\today}

\begin{document}

\setcounter{tocdepth}{2}

\maketitle


\begin{abstract}
  In energy management, it is common that strategic investment
  decisions (storage capacity, production units) are made at a slow
  time scale, whereas operational decisions (storage, production)
  are made at a fast time scale: for such problems, the total number
  of decision stages may be huge. In this paper, we consider multistage
  stochastic optimization problems with two time-scales, and we propose
  a time block decomposition scheme to address them numerically.
  More precisely, our approach relies on two assumptions. On the one hand,
  we suppose slow time scale stagewise independence of the noise process:
  the random variables that occur during a slow time scale interval
  are independent of those at another slow time scale interval.
  This makes it possible to use \DP\ at the slow time scale.
  On the other hand, we suppose a dynamically monotone property
  for the problem under consideration, which makes it possible
  to obtain bounds.
  Then, we present two algorithmic methods to compute upper and
  lower bounds for slow time scale Bellman value functions. Both
  methods rely respectively on primal and dual decomposition of
  the Bellman equation applied at the slow time scale. We assess
  the methods tractability and validate their efficiency by solving
  a battery management problem where the fast time scale operational
  decisions have an impact on the storage current capacity, hence
  on the strategic decisions to renew the battery at the slow time scale.
\end{abstract}

\section{Introduction, motivation and context}

In this paper we present a two-time-scale stochastic optimal
control formalism to optimize systems with fast controlled
stochastic dynamics (e.g. a change every fifteen or thirty minutes)
that affect a controlled long term stochastic behavior (e.g.
a change every day or every week). Using well-known results
from discrete time stochastic optimal control and convex analysis,
we develop two algorithmic methods to decompose that kind
of problems and to design tractable numerical algorithms,
under a suitable monotonicity property.

The starting point of both methods is to apply Bellman principle
at the slow time scale, introducing Bellman value functions that
are ultimately final costs for a sequence of problems (one problem
for each slow time stage) with fast time scale dynamics.
This requires a Markovian assumption at the slow time scale, but
uncertainties impacting the system at the fast time scale may
remain stagewise dependent.

Computing the slow time scale Bellman value functions is numerically
challenging. However, the two algorithms that we present make
it possible to compute upper and lower bounds when the problem
displays a monotonicity property (which is common in energy storage
management problems). We leverage as well periodicity properties,
that are common when dealing with phenomena correlated to seasons,
like energy consumption and like renewable energy production.

We apply the developed algorithms to a battery management problem
to illustrate how our algorithms prove efficient to compute tight
bounds on the slow time scale Bellman value functions.

\subsection{Motivation: a long term battery management problem}
\label{tts:ssect-motivation}

As a motivation of the research, we present an energy storage
management problem over $20$~years. We manage the charge
and discharge of an energy storage every time step~$m$ of
$30$~minutes. A decision of  battery replacement is taken
every day~$D$, so that the number of days considered in the
problem is~$D= 20 \times 365 = 7300$. Since the number of time
steps during a day is~$M = 24 \times 2 = 48$, the total number
of time steps of the problem is~$48 \times 7300 = 350,\!400$.
The state of charge of the battery has to remain between prescribed
bounds at each time step. We also consider the evolution over time
of the amount of remaining exchangeable energy in the battery
(related to the number of cycles remaining). Once this variable
reaches zero, the battery is considered unusable.
In addition to the battery, the studied system includes a local
renewable energy production unit and a local energy consumption:
the net demand (consumption minus production) at each time step
is an exogenous random variable affecting the system.
Finally we pay for the local system energy consumption,
that is, net demand minus energy exchanged with the battery.
When this quantity is negative (excessive energy production),
the energy surplus is assumed to be wasted. The aim of the problem
is to minimize the energy bill over the whole time
horizon by providing an optimal strategy for the storage charge
and the battery renewal.

In this paper, we propose a framework to formally define stochastic
optimization problems naturally displaying two time-scales,
that is, a slow time scale (here days) and a fast time
scale (here half hours). The motivational problem displays
charge/discharge decisions at the fast time scale, that would
be called operational decisions in~\cite{KautEA12}. It also includes
renewal/maintainance of the battery at the slow time scale as well.
The battery renewal decisions correspond to strategic decisions
in~\cite{KautEA12}. The ultimate goal of the paper is to design
tractable algorithms for such problems with hundreds of thousands
of time steps and possibly two decision time scales, without
a stationary/infinite horizon assumption (contrary
to~\cite{tts:haessig:hal-01147369}) and in a stochastic setting
(which extends~\cite{tts:heymann:hal-01349932}).

\subsection{Literature review}

\SDP\ methods (SDP) based on the Bellman equation \cite{tts:bellman57}
is a standard method to solve a stochastic dynamic optimization
optimization problem by time decomposition. This method suffers
the so called~\emph{curses of dimensionality} as introduced
in~\cite{tts:bellman57,tts:bertsekas1995dynamic,tts:powell2007approximate}.
In particular the complexity of the most classical implementation
of SDP (that discretizes the state space) is exponential
in the number of state variables.

A major contribution to handle a large number of state variables is the
well-known \SDDP\ (SDDP) algorithm~\cite{tts:Pereira:1991:MSO:3113604.3113829}.
This method is adapted to problems with linear dynamics and convex costs.
Other similar methods have been developed such as Mixed Integer Dynamic
Approximation Scheme (MIDAS)~\cite{tts:Philpott2016MIDASA} or \SDDIP\
(SDDiP)~\cite{tts:Zou2018} for nonconvex problems,
in particular those displaying binary variables. The performance
of these algorithms is sensitive to the number of time steps
\cite{tts:leclere2018exact,tts:Philpott2016MIDASA}.

Other classical stochastic optimization methods are even more
sensitive to the number of time stages. It is well-known that
solving a multistage stochastic optimization problem on a scenario
tree displays a complexity exponential in the number of time
steps. 

Problems displaying a large number of time stages, in particular
problems with multiple time scales, require to design specific methods.
A class of stochastic optimization problems to deal with two time
scales has been introduced in~\cite{KautEA12} and further formalized
in~\cite{Maggioni2019BoundsIM}. It is called Multi-Horizon Stochastic
Optimization Problems and it frames problems where uncertainty
is modeled using multi-horizon scenario trees as rigorously studied
in~\cite{Werner2013}. Several authors have studied stochastic
optimization problems with interdependent strategic/operational
decisions or intrastage/interstage problems
\cite{abgottspon2016multi,phdthesisAbgottspon,abgottsponmedium,skar2016multi,pritchard2005hydroelectric},
but most of the time the developed methods to tackle the difficulties
are problem dependent. In~\cite{KautEA12} the authors present different
particular cases where the two time-scales (called operational and
strategic decision problems) can be easily decomposed.
In~\cite{Maggioni2019BoundsIM} a formal definition of a Multi Horizon
Stochastic Program is given and methods to compute bounds are developed.
The formal Multi Horizon Stochastic Program is a stochastic optimization
problem with linear cost and dynamics where uncertainties are modeled
as multi time scale scenario trees.

\subsection{Paper contributions}

In this paper, we present the setting of a generic two-time-scale
multistage stochastic optimization problem and the application of Bellman
principle at the slow time scale in this framework. The resulting \DP\
equation is referred to as the Bellman equation by time
blocks, and is detailed in~\cite[Sect.~5]{tts:carpentier:JOCA}.
We introduce the class of dynamically monotone problems that
makes possible to relax the problem dynamics without changing its
optimal value (Sect.~\ref{tts:sec:dp}).
We then devise two algorithms to perform a decomposition of the
problem by time block. The first algorithm, akin to the so-called
price decomposition, always gives a lower bound of the optimal value
of the problem, whereas the second algorithm, based on resource
decomposition, gives an upper bound. This upper bound is relevant,
that is, not almost surely equal to~$+\infty$, for dynamically
monotone multistage stochastic optimization problems. In order to practically
compute these two bounds, we assume that the problem presents
periodicity properties at the slow time scale and we show how to
organize the computation to take advantage of it (Sect.~\ref{tts:sec:bounds}).
We present an application of this method to a battery
management problem incorporating a very large number of time steps
(Sect.~\ref{tts:sec:experiments}). We finally give some insights
on the numerical complexity of the decomposition methods
(Appendix~\ref{tts:ann-complexity}).

\section{Time block decomposition of two-time-scale stochastic optimization problems}
\label{tts:sec:dp}

We present a formal definition of a dynamically monotone
two-time-scale stochastic optimization problem, that is,
a slow time scale and a fast time scale. Then, we introduce
Bellman value functions at the slow time scale as a way to
decompose a two-time-scale stochastic optimization problem
in time blocks.

\subsection{Notations for two time-scales}
\label{ssect:tts-notations}

To properly handle two time-scales, we adopt the following notations.
For a given constant time interval~$\Delta t >0$, let~$M \in \NN^*$
be such that~$(M+1)$ is the number of time steps in the slow time step,
e.g. for~$\Delta t = 30$ minutes,~$M+1 = 48$, when the slow time
step correspond to a day. Given two natural numbers~$r \leq s$, we use
either the notation~$\timeset{r}{s}$ or the notation~$\timesetbis{r}{s}$
for the set $\{r,r+1,\dots,s-1,s\}$. A decision-maker has to make decisions
on two time-scales over a given number of slow time steps~$(D+1)\in \NN^*$:
\begin{enumerate}
\item one type of (say, operational) decision every fast time
  step~$m \in \shortset$ of every slow time step~$d \in \longset$,
\item another type of (say, strategic) decision every slow time
  step~$d \in \longset\cup\{\longfinal\}$.
\end{enumerate}
In our model the time flows between two slow time steps $d$ and $d+1$ as follows:
\begin{align*}
  d,0   \quad \xrightarrow[\;\;\; \Delta t \;\;\;]{} \quad
  d,1   \quad \xrightarrow[\;\;\; \Delta t \;\;\;]{} \quad
  \dots \quad \xrightarrow[\;\;\; \Delta t \;\;\;]{} \quad
  d,M   \quad \xrightarrow[\;\;\; \Delta t \;\;\;]{} \quad
  d+1,0
\end{align*}
A variable~$z$ will have two time indexes~$z_{d,m}$ if it changes every fast
time step~$m$ of every slow time step~$d$. An index~$(d,m)$ belongs to the set
\begin{equation}
  \TT = \longset \times \shortset \cup \{ \np{\longfinal,0} \} \eqfinv
  \label{tts:eq:couples_of_indices}
\end{equation}
which is a totally ordered set when equipped with the lexicographical order~$\preceq$:
\begin{equation}
  (d,m) \preceq (d',m') \iff   \np{d \leq d'} \text{ or }
  \bp{ d= d' \text{ and } m \leq m'} \eqfinp
  \label{tts:eq:lexicographically_ordered_time_span}
\end{equation}
We also use the following notations for describing sequences of
variables and sequences of spaces. For $(d,m)$ and $(d,m') \in \TT$, with
$m \le m'$:
\begin{itemize}
\item the notation $z_{d,m:m'}$ refers to the sequence
  of variables $\np{ z_{d,m},z_{d,m+1},\dots, z_{d,m'-1},z_{d,m'}}$,
\item the notation $\ZZ_{d,m:m'}$ refers to the cartesian product
  of spaces~$\prod_{k=m}^{m'} \ZZ_{d,k}$.
\end{itemize}

\subsection{Two-time-scale multistage stochastic optimization setting}

We consider a probability space~$(\omeg,\trib,\prbt)$ and an exogenous
noise process~$\va{w}=\na{\va{w}_{d}}_{d\in\longset}$ at the slow time
scale. Random variables are denoted using bold letters,
and we denote by \( \sigma\np{\va{Z}} \) the $\sigma$-algebra generated by the
random variable~$\va{Z}$.
For any $d\in\longset$, the random variable~$\va{w}_{d}$ consists of
a sequence of random variables~$\na{\va{w}_{d,m}}_{m\in\shortset}$
at the fast time scale:
\begin{equation}
  \label{tts:eq:splatWd}
  \va{w}_d = (\va{w}_{d,0}, \dots, \va{w}_{d,m}, \dots,\va{w}_{d,M}) \eqfinp
\end{equation}
Each random variable~$\va{w}_{d,m}:\omeg \to \WW_{d,m}$ takes values
in a measurable space $\WW_{d,m}$ (``uncertainty'' space) equipped
with a $\sigma$-field $\tribu{\Uncertainty}_{d,m}$, so that
$\va{w}_{d}:\omeg\to\WW_{d}$ takes values in the product space
$\WW_{d} = \WW_{d,0:M} = \prod_{m=0}^{M} \WW_{d,m}$.
\begin{subequations}
  \label{tts:eq:tribus}
  For any $(d,m)\in\TT$, we denote by~$\tribu{F}_{d,m}$ the $\sigma$-field
  generated by all noises up to time~$(d,m)$, that is,
  \begin{equation}
    \tribu{F}_{d,m} =
    \sigma \bp{\va{w}_0,\dots,\va{w}_{d-1},\va{w}_{d,0},\dots,\va{w}_{d,m}} \eqfinp
  \end{equation}
  We also introduce the filtration~$\filtration$ at the slow time scale:
  \begin{equation}
    \filtration = \{\tribu{F}_{d,M}\}_{d\in\longset}
    = \bp{ \tribu{F}_{0,M}, \ldots, \tribu{F}_{D,M} } \eqfinp
  \end{equation}
\end{subequations}

In the same vein, we introduce a decision
process~$\va{u}=\na{\va{u}_{d}}_{d\in\longset}$ at the slow time scale,
where each~$\va{u}_{d}$ consists of a sequence
$\na{\va{u}_{d,m}}_{m\in\shortset}$ of decision variables at the fast time scale:
\begin{equation}
  \label{tts:eq:splatUd}
  \va{u}_d = (\va{u}_{d,0}, \ldots, \va{u}_{d,m}, \ldots,\va{u}_{d,M}) \eqfinp
\end{equation}
Each random variable~$\va{u}_{d,m}:\omeg \to \UU_{d,m}$ takes values in
a measurable space $\UU_{d,m}$ (``control'' space) equipped with a
$\sigma$-field $\tribu{\Control}_{d,m}$, and we denote by~$\UU_{d}$
the cartesian product $\UU_{d,0:M}$. We finally introduce a state
process~$\va{x}=\na{\va{x}_{d}}_{d\in\longset\cup\{\longfinal\}}$ at the
slow time scale, where each random variable~$\va{x}_{d}:\Omega \to \XX_{d}$
takes values in a measurable space~$\XX_{d}$ (``state'' space) equipped with
a $\sigma$-field $\tribu{\State}_{d}$. Note that, unlike processes~$\va{w}$
and~$\va{u}$, \emph{the state process~$\va{x}$ is defined only at the slow time
scale}. Thus, for any $d\in\longset \cup \{\longfinal\}$, the random
variable $\va{x}_{d}$ represents  the system state at time~$(d,0)$.

We also consider measurable spaces~$\pricespace_{d}$ such that,
for each~$d$, $\XX_{d}$ and~$\pricespace_{d}$ are paired spaces
when equipped with a bilinear form~$\nscal{\cdot}{\cdot}$.
In this paper, we assume that each state space~$\XX_{d}$ is
the vector space~$\RR^{n_{d}}$, so that $\pricespace_{d}=\RR^{n_{d}}$,
the bilinear form~$\nscal{\cdot}{\cdot}$ being the standard scalar product.

For each $d\in\longset$, we introduce a measurable instantaneous cost function
$\coutint_{d} : \XX_{d}\times\UU_{d}\times\WW_{d}\rightarrow \ClosedIntervalClosed{0}{+\infty}$
and a measurable dynamics $\dynamics_{d} :
\XX_{d}\times\UU_{d}\times\WW_{d}\rightarrow \XX_{d+1}$.
Note that both the instantaneous cost~$\coutint_{d}$ and the
dynamics~$\dynamics_{d}$ depend on all the fast time scale decision
and noise variables constituting the slow time step~$d$.
We also introduce a measurable final cost function
$\coutfin : \XX_{\longfinal}\rightarrow \ClosedIntervalClosed{0}{+\infty}$.

With all these ingredients, we write a two-time-scale stochastic optimization
problem:\footnote{The notation~$V^{\mathrm{e}}(x)$ for the optimal value
of Problem~\eqref{tts:eq:2tspb} emphasizes the fact that the dynamics
equations~\eqref{tts:eq:2tsdyn} correspond to equality constraints.
We will consider later Problem~\eqref{tts:eq:2tspbmon}, subject
to dynamics \emph{inequality constraints}, and its optimal value
will thus be denoted by~$V^{\mathrm{i}}(x)$.}
\begin{subequations}
  \label{tts:eq:2tspb}
  \begin{align}
    V^{\mathrm{e}}(x) = \inf_{\va{x},\va{u}} \;
    & \EE \bgc{\sum_{d=0}^{D} \coutint_d(\va X_d,\va U_d,\va W_{d}) +
      \coutfin(\va X_{\longfinal})}
      \label{tts:eq:2tscost} \eqfinv \\
    \text{s.t} \quad
    & \va X_{0} = x \eqsepv
      \va X_{d+1} = \dynamics_{d}(\va X_d,\va U_d,\va W_{d}) \eqsepv \forall d \in\longset
      \label{tts:eq:2tsdyn} \eqfinv \\
    & \sigma(\va{u}_{d,m}) \subset \tribu{F}_{d,m}
      \eqsepv \forall (d,m)\in\longset\times\shortset
      \label{tts:eq:2tsnonant} \eqfinp
  \end{align}
\end{subequations}
The expected cost value in~\eqref{tts:eq:2tspb} is well defined,
as all functions are nonnegative measurable.
Constraint~\eqref{tts:eq:2tsnonant} --- where $\sigma(\va{u}_{d,m})$
is the $\sigma$-field generated by the random variable~$\va{u}_{d,m}$
--- expresses the fact that each decision $\va{u}_{d,m}$ is
$\tribu{F}_{d,m}$-measurable, that is, is nonanticipative.
\begin{remark}
  \label{rm:constraints}
  We just consider as explicit constraints dynamic constraints
  as~\eqref{tts:eq:2tsdyn} and nonanticipativity constraints
  as~\eqref{tts:eq:2tsnonant}, but other constraints involving
  the state and the control can be incorporated in the instantaneous
  cost $\coutint_{d}$ or in the final cost $\coutfin$ by means of
  indicator functions as~$\coutint_{d}$ and $\coutfin$ can take
  values in the extended real numbers space $\RR \cup \na{+\infty}$.
\end{remark}

Problem~\eqref{tts:eq:2tspb} seems very similar to a classical
discrete time multistage stochastic optimization problem.
But an important difference appears in the nonanticipativity
constraint~\eqref{tts:eq:2tsnonant} that expresses the fact that
the decision vector~$\va U_d = (\va U_{d,0},\dots,\va U_{d,M})$
at every slow time step~$d$ does not display the same measurability
for each component (information grows every fast time step).
This point of view is not referred to in the literature and is one
of the novelty of the paper.

\subsection{Dynamically monotone multistage stochastic optimization problems}
\label{ssect:2tspbmon}

Here, we introduce the notion of dynamically monotone multistage
stochastic optimization problem.

For this purpose, we consider the following multistage stochastic
optimization problem:
\begin{subequations}
  \label{tts:eq:2tspbmon}
  \begin{align}
    V^{\mathrm{i}}(x) = \inf_{\va{x},\va{u}} \;
    & \EE \bgc{\sum_{d=0}^{D} \coutint_{d}(\va X_d,\va U_d,\va W_{d}) +
      \coutfin(\va X_{D+1})} \eqfinv \\
    \text{s.t} \quad
    & \va X_{0} = x \eqsepv
      \dynamics_{d}(\va X_d,\va U_d,\va W_{d}) \geq \va X_{d+1} \eqsepv
      \forall d\in\longset \label{tts:eq:2tsdynmon} \eqfinv \\
    & \sigma(\va{u}_{d,m}) \subset \tribu{F}_{d,m}
      \eqsepv \forall (d,m)\in\longset\times\shortset
      \label{tts:eq:2tsnonantmon} \eqfinv \\
    & \sigma(\va{x}_{d+1}) \subset \tribu{F}_{d,M}
      \eqsepv \forall d\in\longset \eqfinp
  \end{align}
\end{subequations}
We have relaxed the dynamic equality constraints~\eqref{tts:eq:2tsdyn}
into inequality constraints~\eqref{tts:eq:2tsdynmon}.
Thus, Problem~\eqref{tts:eq:2tspbmon} is less constrained
than Problem~\eqref{tts:eq:2tspb},
so that the optimal value~$V^{\mathrm{i}}(x)$
of Problem~\eqref{tts:eq:2tspbmon} is less than the optimal
value~$V^{\mathrm{e}}(x)$ of Problem~\eqref{tts:eq:2tspb}:
\begin{equation*}
V^{\mathrm{i}}(x) \leq V^{\mathrm{e}}(x) \eqsepv \forall x \in \XX_{0} \eqfinp
\end{equation*}

\begin{definition}\label{tts:hyp:monotonicity}
  The problem~\eqref{tts:eq:2tspb} is said to be \emph{dynamically monotone}
  if it is equivalent to the relaxed problem~\eqref{tts:eq:2tspbmon}
  in the sense that they have the same optimal values:
  $V^{\mathrm{e}}(x) = V^{\mathrm{i}}(x)$ for any given initial
  state~$x \in \XX_{0}$.  Otherwise stated, in Problem~\eqref{tts:eq:2tspb}, the
  dynamic equality constraints~\eqref{tts:eq:2tsdyn} can be replaced by the
  dynamic inequality constraint~\eqref{tts:eq:2tsdynmon} without changing the
  optimal value.
\end{definition}

In this paper, we are interested in methods to solve dynamically
monotone two-time-scale multistage stochastic optimization problems.

\subsection{Bellman's principle of optimality at the slow time scale}
\label{ssect:bellman-2level}

\SDP, based on Bellman optimality principle, is a classical way to decompose
multistage stochastic optimization problems into multiple but smaller static
optimization problems. In this paragraph, we apply a Bellman optimality
principle by time blocks to decompose the multistage two-time-scale stochastic
optimization problem~\eqref{tts:eq:2tspbmon}, into multiple smaller problems
that are stochastic optimization problems over a single slow time step.

We first introduce a sequence
$\na{V_{d}^{\mathrm{e}}}_{d\in\longset\cup\{\longfinal\}}$
of~\emph{slow time scale Bellman value functions}
associated with Problem~\eqref{tts:eq:2tspb}.
These functions are defined by backward induction as follows.
At time $\longfinal$, we set~$V_{\longfinal}^{\mathrm{e}} = \coutfin$,
and then, for $d\in\longset$ and for all~$x\in\XX_{d}$,
\begin{subequations}
  \label{tts:bellman-2leveleq}
  \begin{align}
    V_{d}^{\mathrm{e}}(x) = \inf_{\va{x}_{d+1},\va{u}_{d}} \;
    & \EE \; \Bc{ \coutint_{d}(x,\va{u}_d,\va{w}_{d}) + V_{d+1}^{\mathrm{e}}(\va{x}_{d+1})}
      \eqfinv \\
    \text{s.t} \quad
    & \va{x}_{d+1} = \dynamics_{d}(x,\va{u}_d,\va{w}_{d})
      \label{tts:bellman-2leveleq-dyn} \eqfinv \\
    & \sigma(\va{u}_{d,m}) \subset \sigma(\va{w}_{d,0:m})
      \eqsepv \forall m \in \shortset
      \label{tts:bellman-2leveleq-mes} \eqfinp
  \end{align}
\end{subequations}
We also introduce a sequence of~\emph{slow time scale Bellman value functions}
$\na{V_{d}^{\mathrm{i}}}_{d\in\longset\cup\{\longfinal\}}$ associated
with Problem~\eqref{tts:eq:2tspbmon}. At time $\longfinal$,
we set~$V_{\longfinal}^{\mathrm{i}} = \coutfin$, and then,
for $d\in\longset$ and for all~$x\in\XX_{d}$,
\begin{subequations}
  \label{tts:bellman-2level}
  \begin{align}
    V_{d}^{\mathrm{i}}(x) = \inf_{\va{x}_{d+1},\va{u}_{d}} \;
    & \EE \; \Bc{ \coutint_{d}(x,\va{u}_d,\va{w}_{d}) + V_{d+1}^{\mathrm{i}}(\va{x}_{d+1})}
      \eqfinv \\
    \text{s.t} \quad
    & \dynamics_{d}(x,\va{u}_d,\va{w}_{d}) \geq \va{x}_{d+1}
      \label{tts:bellman-2level-dyn} \eqfinv \\
    & \sigma(\va{u}_{d,m}) \subset \sigma(\va{w}_{d,0:m})
      \eqsepv \forall m \in \shortset
      \label{tts:bellman-2level-mes} \eqfinv \\
    & \sigma(\va{x}_{d+1}) \subset \sigma(\va{w}_{d,0:M})
      \label{tts:bellman-2level-mesX} \eqfinp
  \end{align}
\end{subequations}
Problem~\eqref{tts:bellman-2level} is less constrained than
Problem~\eqref{tts:bellman-2leveleq} because the (dynamics) equality
constraints~\eqref{tts:bellman-2leveleq-dyn} are more binding
than the inequality constraints~\eqref{tts:bellman-2level-dyn},
and also because~\eqref{tts:bellman-2leveleq-dyn}
implies~\eqref{tts:bellman-2level-mesX}.
Since~$V_{\longfinal}^{\mathrm{e}} = V_{\longfinal}^{\mathrm{i}} = \coutfin$,
we obtain by backward induction that the solution
of Problem~\eqref{tts:bellman-2level} gives a lower bound to
the solution of Problem~\eqref{tts:bellman-2leveleq}:
\begin{equation}
  \label{tts:eq:lessconstrained}
  V_{d}^{\mathrm{i}} \leq V_{d}^{\mathrm{e}}
  \eqsepv
  \forall d\in\longset\cup\{\longfinal\}
  \eqfinp
\end{equation}

We introduce a specific independence assumption for the noise
process~$\va{w}$ in order to apply the Bellman principle at the
slow time scale for the optimization problem~\eqref{tts:eq:2tspb},
that is, a \SDP\ equation at the slow time scale.
\begin{assumption}(White noise assumption)
  \label{tts:hyp:indep}
  The sequence~$\sequence{\va{w}_d}{d\in\longset}$ is white, that is,\\
  $\ba{\bp{\va{w}_{d,0},\dots,\va{w}_{d,m},\dots,\va{w}_{d,M} }}_{d\in\longset}$
  is a sequence of independent random vectors.
\end{assumption}
This assumption is made throughout the rest of the paper.

\begin{remark}
  We do not assume that each random vector
  $\va{w}_d = (\va{w}_{d,0},\dots,\va{w}_{d,M})$ is itself composed
  of independent random variables.
\end{remark}

\begin{proposition}
  \label{prop:dailyBellmanPrinciple}
  Under the white noise Assumption~\ref{tts:hyp:indep}, the value
  function~$V^{\mathrm{e}}$ (resp. $V^{\mathrm{i}}$) solution
  of Problem~\eqref{tts:eq:2tspb} (resp. solution of the relaxed
  problem~\eqref{tts:eq:2tspbmon}) coincides with the Bellman value
  function~$V_0^{\mathrm{e}}$ (resp. $V_0^{\mathrm{i}}$)
  given by Bellman equations~\eqref{tts:bellman-2leveleq}
  (resp.~\eqref{tts:bellman-2level}). More explicitly, we have that
  \begin{equation*}
    V^{\mathrm{e}}(x) = V_{0}^{\mathrm{e}}(x) \quad \text{and} \quad
    V^{\mathrm{i}}(x) = V_{0}^{\mathrm{i}}(x)
    \eqsepv \forall x \in \XX_{0} \eqfinp
  \end{equation*}
\end{proposition}

\begin{proof}
  The fact that the function $V^{\mathrm{e}}$ is equal
  to the function $V_{0}^{\mathrm{e}}$ is a consequence
  of~\cite[Proposition~10]{tts:carpentier:JOCA} where
  the machinery for establishing a \DP\ equation in
  a two-time-scale  multistage stochastic optimization
  setting is developed.
  To establish the equality between the functions~$V^{\mathrm{i}}$
  and $V_{0}^{\mathrm{i}}$, we proceed as follows. First, it is easily
  established that Problem~\eqref{tts:eq:2tspbmon} is equivalent to
  Problem~\eqref{tts:eq:2tspbmonmod} stated below which involves a new
  decision process~$\va{\Delta}=\na{\va{\Delta}_{d}}_{d\in\longset}$,
  each control variable~$\va{\Delta}_{d}$ taking values in~$\XX_{d+1}$:
  \begin{subequations}
    \label{tts:eq:2tspbmonmod}
    \begin{align}
      V^{\mathrm{i}}(x) = \inf_{\va{x},\va{u},\va{\Delta}} \;
      & \EE \bgc{\sum_{d=0}^{D} \coutint_{d}(\va{x}_d,\va{u}_d,\va{w}_{d}) +
        \coutfin(\va{x}_{D+1})} \eqfinv \\
      \text{s.t} \quad
      & \va{x}_{0} = x \eqsepv
        \va{x}_{d+1} = \dynamics_{d}(\va{x}_d,\va{u}_d,\va{w}_{d}) - \va{\Delta}_{d}
        \eqsepv \forall d\in\longset \label{tts:eq:2tsdynmonmod} \eqfinv \\
      & \va{\Delta}_{d} \geq 0 \quad \forall d\in\longset \eqfinv \\
      & \sigma(\va{u}_{d,m}) \subset \tribu{F}_{d,m}
        \eqsepv \forall (d,m)\in\longset\times\shortset
        \label{tts:eq:2tsnonantmonmod} \eqfinv \\
      & \sigma(\va{\Delta}_{d}) \subset \tribu{F}_{d,M}
        \eqsepv \forall d\in\longset \eqfinp
    \end{align}
  \end{subequations}
  {Second, Problem~\eqref{tts:eq:2tspbmonmod} involves standard
  equality constraints in the dynamics, so that the machinery
  developed in~\cite[Propositions~10]{tts:carpentier:JOCA}
  applies to it.} We therefore obtain a \DP\ equation associated
  with Problem~\eqref{tts:eq:2tspbmonmod} involving the new decision
  process~$\va{\Delta}$. This last \DP\ equation reduces to the
  Bellman equation~\eqref{tts:bellman-2level} by replacing the extra
  nonnegative decision variables by inequality constraints.
\end{proof}

We also introduce a monotonicity assumption for the sequence of Bellman value
functions $\na{V_{d}^{\mathrm{e}}}_{d \in\longset\cup\{\longfinal\}}$ to ensure
that the dynamically monotone property of Definition~\ref{tts:hyp:monotonicity}
is fulfilled.

\begin{assumption}(Nonincreasing Bellman value functions assumption)
  \label{tts:hyp:nonincreasing}
  The Bellman value
  functions~$\na{V_{d}^{\mathrm{e}}}_{d\in \longset\cup\{\longfinal\}}$
  solutions of the Bellman equation~\eqref{tts:bellman-2leveleq} are
  nonincreasing functions of their arguments:
  \begin{equation*}
    \forall d \in \longset\cup\{\longfinal\} \eqsepv
    \forall (x,x') \in \XX_{d}\times\XX_{d} \eqsepv
    x \leq x' \; \Longrightarrow \;
    V_{d}^{\mathrm{e}}(x) \geq V_{d}^{\mathrm{e}}(x') \eqfinp
  \end{equation*}
\end{assumption}
Although this Assumption~\ref{tts:hyp:nonincreasing} has to do with monotonicity
(with respect to the argument of Bellman value functions),
it has nothing to do with the dynamically monotone notion in
Definition~\ref{tts:hyp:monotonicity}.

\begin{proposition}
  \label{prop:dailymonotone}
  We suppose that the Nonincreasing Bellman value functions Assumption~\ref{tts:hyp:nonincreasing} holds true. Then, for
  any~$d\in\longset\cup\{\longfinal\}$, the (original) Bellman value
  function~$V_{d}^{\mathrm{e}}$ in~\eqref{tts:eq:2tspb} coincides with
  the (relaxed) Bellman value function~$V_{d}^{\mathrm{i}}$
  in~\eqref{tts:eq:2tspbmon}:
  \begin{equation*}
    V_{d}^{\mathrm{i}} = V_{d}^{\mathrm{e}} \eqsepv
    \forall d\in\longset\cup\{\longfinal\} \eqfinp
  \end{equation*}
\end{proposition}

\begin{proof}
  By Equation~\eqref{tts:eq:lessconstrained}, we have
  that~$V_{d}^{\mathrm{i}} \leq V_{d}^{\mathrm{e}}$ for
  all~$d\in\longset\cup\{\longfinal\}$. To obtain the reverse
  inequality, we proceed by backward induction.
  First, at time $D+1$, the two functions~$V_{\longfinal}^{\mathrm{e}}$ and $V_{\longfinal}^{\mathrm{i}}$
  are both equal to the function~$\coutfin$.
  Second, let $d$ be fixed in $\longset$ and assume that
  $V_{d+1}^{\mathrm{i}}\ge V_{d+1}^{\mathrm{e}}$.
  For any~$\epsilon > 0$, let~$(\va{x}_{d+1},\va{u}_{d})$ be an
  $\epsilon$-optimal solution of Problem~\eqref{tts:bellman-2level}.
  Then, we have that
  \begin{align*}
    V_{d}^{\mathrm{i}}(x) + \epsilon
    & \geq \;
      \EE \; \bc{ \coutint_{d}(x,\va{u}_d,\va{w}_{d})
      + V_{d+1}^{\mathrm{i}}(\va{x}_{d+1})} \eqfinv \\
    & \geq \;
      \EE \; \bc{ \coutint_{d}(x,\va{u}_d,\va{w}_{d})
      + V_{d+1}^{\mathrm{e}}(\va{x}_{d+1})} \eqfinv
      \tag{by induction assumption $V_{d+1}^{\mathrm{i}}\ge V_{d+1}^{\mathrm{e}}$} \\
    & \geq  \;
      \EE \; \bc{ \coutint_{d}(x,\va{u}_d,\va{w}_{d})
      + V_{d+1}^{\mathrm{e}}\bp{\dynamics_{d}(x,\va{u}_d,\va{w}_{d})}}
      \tag{by monotonicity of~$V_{d+1}^{\mathrm{e}}$ and \eqref{tts:bellman-2level-dyn} }
      \eqfinv \\
    & \geq V_{d}^{\mathrm{e}}(x)
      \tag{by definition~\eqref{tts:bellman-2leveleq} of~$V_{d}^{\mathrm{e}}(x)$,
      since $\va{u}_{d}$ is admissible for Problem~\eqref{tts:bellman-2leveleq}}
      \eqfinp
  \end{align*}
  We thus obtain the reverse inequality
  $V_{d}^{\mathrm{i}}\ge V_{d}^{\mathrm{e}}$ for
  all~$d\in\longset\cup\{\longfinal\}$.
\end{proof}

As an immediate consequence of Propositions~\ref{prop:dailyBellmanPrinciple}
and~\ref{prop:dailymonotone}, we obtain the following proposition 
which is the main result of this section.

\begin{proposition}
\label{pr:main_result_of_Section2}
  Suppose that the white noise Assumption~\ref{tts:hyp:indep} and
the Nonincreasing Bellman value functions
Assumption~\ref{tts:hyp:nonincreasing} hold true.
Then, Problem~\eqref{tts:eq:2tspb} is dynamically monotone and its
  optimal value function~$V^{\mathrm{e}}$ can be computed by solving
  Problem~\eqref{tts:eq:2tspbmon} at the slow time scale by the Bellman
  backward induction~\eqref{tts:bellman-2level}, that is,
  \begin{equation*}
    V^{\mathrm{i}}(x) = V_{0}^{\mathrm{i}}(x) = V_{0}^{\mathrm{e}}(x)
    = V^{\mathrm{e}}(x) \eqsepv \forall x\in\XX_{0} \eqfinp
  \end{equation*}
\end{proposition}

The issue is that performing the backward induction~\eqref{tts:bellman-2level}
requires to solve $D$~multi-stage stochastic optimization problems at the fast
time scale. In the next section, we present two methods to compute bounds
of the Bellman value functions at the slow time scale, that allow to leverage
periodicity of the problem and to simplify the backward induction.

\section{Bounds for Bellman value functions of dynamically monotone problems}
\label{tts:sec:bounds}

As seen in~\S\ref{ssect:bellman-2level}, under
Assumptions~\ref{tts:hyp:indep} and~\ref{tts:hyp:nonincreasing},
the optimal value functions of Problem~\eqref{tts:eq:2tspb}
can be computed by solving Problem~\eqref{tts:eq:2tspbmon}
at the slow time scale by the Bellman backward induction~\eqref{tts:bellman-2level}.
We aim at finding tractable algorithms to numerically solve the backward
induction~\eqref{tts:bellman-2level} and obtain the sequence of Bellman
value functions $\{V_{d}^{\mathrm{i}}\}_{d\in\longset\cup\{\longfinal\}}$.
Indeed, these Bellman value functions are not easily obtained. The main
issue is that the optimization problem~\eqref{tts:bellman-2level} is
a multistage stochastic optimization problem that has to be solved
for every~$d \in \longset$ and every~$x \in \XX_d$, and each numerical
solving might be in itself hard.

To tackle this issue, we propose in~\S\ref{Lower_bounds_of_the_value_functions}
and~\S\ref{Upper_bounds_of_the_value_functions} to compute respectively
lower and upper bounds of the Bellman value functions at the slow time
scale. These Bellman value functions bounds can then be used to design
admissible two-time-scale optimization policies. The two algorithms are
based on so-called price and resource decomposition techniques (see
\cite[Chap.~6]{tts:bertsekas1999nonlinear} and \cite{tts:carpentier2017decomp})
applied to Problem~\eqref{tts:bellman-2level}.

Both algorithms involve the computation of auxiliary functions
that gather the fast time scale computations, and that are
numerically appealing because they allow to exploit some eventual
periodicity of two-time-scale problems and parallel computation.
This point is developed in~\S\ref{subsec:periodicity}.

\subsection{Lower bounds of the Bellman value functions}
\label{Lower_bounds_of_the_value_functions}

We present lower bounds for the Bellman value functions $\ba{V_{d}^{\mathrm{i}}}_{d\in\longset\cup\{\longfinal\}}$
given by Equation~\eqref{tts:bellman-2level}. These bounds derive
from an algorithm which appears to be connected to the one
developed in~\cite{tts:heymann:hal-01349932}, called
``adaptive weights algorithm". We extend the results of~\cite{tts:heymann:hal-01349932}
in a stochastic setting and in a more general framework, as we are not tied to a battery
management problem and as we use a more direct way to reach similar
conclusions.

To obtain lower bounds of the
sequence~$\{V_{d}^{\mathrm{i}}\}_{d\in\longset\cup\{\longfinal\}}$ of Bellman
value functions,
we dualize the dynamic equations~\eqref{tts:bellman-2level-dyn}
with Lagrange multipliers, and we use weak duality. The multipliers
(called prices here) could be chosen in the class of nonpositive
$\filtration$-adapted processes but it is enough, to get lower bounds,
to stick to deterministic price processes. Following these lines we obtain
a lower bound as follows.

We define the function
$\coutint_{d}^{\mathrm{P}}:\XX_{d}\times\pricespace_{d+1}\to \RR \cup
\na{\pm\infty}$
(recall that $\pricespace_{d}=\RR^{n_{d}}$), by
\begin{subequations}
  \label{tts:eq:intrapbdual}
  \begin{align}
    \coutint_{d}^{\mathrm{P}}(x_{d},\price_{d+1}) = \inf_{\va{u}_{d}} \;
    & \EE \; \Bc{\coutint_{d}(x_{d},\va{u}_d,\va{w}_{d})
      + \bscal{\price_{d+1}}{\dynamics_{d}(x,\va{u}_d,\va{w}_{d})}} \eqfinv \\
    \text{s.t.} \quad
    & \sigma(\va U_{d,m}) \subset \sigma(\va{w}_{d,0:m})
      \eqsepv \forall m \in \shortset \eqfinv
  \end{align}
\end{subequations}
where~$\coutint_{d}$ and~$\dynamics_{d}$ are respectively the instantaneous
cost function and the dynamics of Problem~\eqref{tts:eq:2tspb}.

Recall that, for a function $g:\XX_{d}\to \RR \cup\na{\pm \infty}$,
  $\LFM{g}: \pricespace_{d+1}\to \RR \cup\na{\pm \infty}$ denotes
  the Fenchel conjugate of~$g$ (see~\cite{tts:rockafellar2015convex}).

\begin{proposition}
  \label{tts:prop:dualintraineq}
Suppose that the white noise Assumption~\ref{tts:hyp:indep} and
the Nonincreasing Bellman value functions
Assumption~\ref{tts:hyp:nonincreasing} hold true.
  Consider the sequence $\ba{\underline{V}_{d}^{\mathrm{P}}}_{d\in\longset\cup\{\longfinal\}}$
of Bellman value functions
  which is defined by $\underline{V}_{\longfinal}^{\mathrm{P}} = \coutfin$
  and for all $d\in\longset$, and for all $x\in\XX_d$
  by\footnote{As we manipulate functions with values in~$\barRR =
      [-\infty,+\infty] $, we need to take care with additions of extended
      reals.
      When not explicitely specified,
    we adopt by default that~$+$ is the Moreau lower
    addition~$\LowPlus$~\cite{Moreau:1970},
    which extends the usual addition to extended reals by
    \( \np{+\infty} + \np{-\infty}=\np{-\infty} + \np{+\infty}=-\infty \).
    Thus Equation~\eqref{tts:eq:bellman-dualized-det} has to be understood as
    \( \underline{V}_{d}^{\mathrm{P}}(x)
    = \sup_{\price_{d+1} \leq 0} \Bp{\coutint_{d}^{\mathrm{P}}(x,\price_{d+1})
      \LowPlus \np{- \LFM{\bp{\underline{V}_{d+1}^{\mathrm{P}}}}(\price_{d+1}) }
      } \).
  }
  \begin{align}
    \label{tts:eq:bellman-dualized-det}
    \underline{V}_{d}^{\mathrm{P}}(x)
    & = \sup_{\price_{d+1} \leq 0} \Bp{\coutint_{d}^{\mathrm{P}}(x,\price_{d+1})
      - \LFM{\bp{\underline{V}_{d+1}^{\mathrm{P}}}}(\price_{d+1}) } \eqfinp
  \end{align}
  Then, the Bellman value functions
  $\ba{\underline{V}_{d}^{\mathrm{P}}}_{\longset\cup\{\longfinal\}}$
  given by Equation~\eqref{tts:eq:bellman-dualized-det} are lower bounds
  of the corresponding Bellman  value functions
  $\ba{V_{d}^{\mathrm{i}}}_{d\in\longset\cup\{\longfinal\}}$ given
  by Equation~\eqref{tts:bellman-2level}, that is,
  \begin{equation}
    \label{tts:inegPi}
    \underline{V}_{d}^{\mathrm{P}} \leq V_{d}^{\mathrm{i}}
    \eqsepv \forall d\in\longset\cup\{\longfinal\} \eqfinp
  \end{equation}
\end{proposition}

\begin{proof}
  We start the proof by a preliminary interchange result. We consider
  a subset $\espacef{X}$ of the space of random variables taking values
  in $\XX$ and we assume that $\espacef{X}$ contains all the constant
  random variables. Then, we prove that
  \begin{equation}
    \label{tts:interchange}
    \inf_{\va{x}\in \espacef{X}} \EE \bc{\varphi(\va{x})} =
    \inf_{x\in \XX} \varphi(x) \eqfinp
  \end{equation}
  \begin{itemize}
  \item
    The $\le$ inequality \( \inf_{\va{x}\in \espacef{X}} \EE \bc{\varphi(\va{x})} \leq
    \inf_{x\in \XX} \varphi(x) \)
    is clear as $\espacef{X}$ contains all
    the constant random variables.
  \item The reverse inequality holds true if
    $\inf_{x\in \XX} \varphi(x) = -\infty$ since
    $\inf_{\va{x}\in X} \EE \bc{\varphi(\va{x})} \leq \inf_{x\in \XX} \varphi(x)$.
    Assume now that $\inf_{x\in \XX} \varphi(x) = \underline{\varphi} > -\infty$. Then
    $\varphi(\va{x}) \geq \underline{\varphi} \;$ \Pps for all $\va{x}\in\espacef{X}$
    and hence $\inf_{\va{x}\in\espacef{X}} \EE \bc{\varphi(\va{x})} \geq \underline{\varphi}$.
    Consider an arbitrary $\epsilon >0$ and $\va{x}_{\epsilon}$ such that
    $\EE \bc{\varphi(\va{x}_{\epsilon})} \le
    \inf_{\va{x}\in X} \EE \bc{\varphi(\va{x})} + \epsilon$.
    We successively obtain
    $\inf_{x\in \XX} \varphi(x) = \EE\bc{ \inf_{x\in \XX} \varphi(x)}
    \le \EE\bc{\varphi(\va{x}_{\epsilon})}
    \le \inf_{\va{x}\in X} \EE \bc{\varphi(\va{x})} + \epsilon$.
    Thus, the reverse inequality \( \inf_{\va{x}\in \espacef{X}} \EE \bc{\varphi(\va{x})} \geq
    \inf_{x\in \XX} \varphi(x) \) follows, hence the equality in \eqref{tts:interchange}.
  \end{itemize}

  We turn now to the proof of~\eqref{tts:inegPi}, that we do
  by backward induction. First, we have that
  $\underline{V}_{D+1}^{\mathrm{P}} = \coutfin = V_{D+1}^{\mathrm{i}}$.
  Second, consider $d\in \longset$ and assume that
  $\underline{V}_{d+1}^{\mathrm{P}} \leq V_{d+1}^{\mathrm{i}}$.
  We successively have \footnote{Here below, we sometimes
  explicitely use~$\LowPlus$ to stress that there might be an addition
  of two conflicting~$\pm\infty$. When we leave the notation~$+$,
  it is because either we sum real numbers or we sum a real number
  with~$\pm\infty$ or we sum elements of $\ClosedIntervalClosed{0}{+\infty}$.}
  \begin{align*}
    \underline{V}_{d}^{\mathrm{P}}(x)
    & = \sup_{\price_{d+1} \leq 0} \Bp{\coutint_{d}^{\mathrm{P}}(x,\price_{d+1})
      \LowPlus \np{- \LFM{\bp{\underline{V}_{d+1}^{\mathrm{P}}}}(\price_{d+1}) }
      }
\intertext{by~\eqref{tts:eq:bellman-dualized-det} where we explicitely use the
      Moreau lower addition~$\LowPlus$}
    & = \sup_{\price_{d+1} \leq 0} \Bp{ \coutint_{d}^{\mathrm{P}}(x,\price_{d+1})
      \LowPlus 
      {\inf_{{x}_{d+1}} \bp{ - \underbrace{ \bscal{\price_{d+1}}{{x}_{d+1}}
      }_{\in\RR, \textrm{hence the following~}+}
      + \underline{V}_{d+1}^{\mathrm{P}}({x}_{d+1})}}}
\tag{by definition of the Fenchel conjugate}
    \\
    & \le
      \sup_{\price_{d+1} \leq 0} \Bp{ \coutint_{d}^{\mathrm{P}}(x,\price_{d+1})
      \LowPlus 
      {\inf_{{x}_{d+1}} \bp{ - \bscal{\price_{d+1}}{{x}_{d+1}}
      + \underline{V}_{d+1}^{\mathrm{i}}({x}_{d+1})}}}
      \tag{by the induction assumption} \\
    & =
      \sup_{\price_{d+1} \leq 0} \Bp{ \coutint_{d}^{\mathrm{P}}(x,\price_{d+1})
      \LowPlus 
      \inf_{\va{x}_{d+1}} \EE \bc{ - \bscal{\price_{d+1}}{\va{x}_{d+1}}
      + \underline{V}_{d+1}^{\mathrm{i}}(\va{x}_{d+1})} }
      \tag{by the interchange result} \\
    & \leq
      \sup_{\price_{d+1} \leq 0} \;
      \inf_{\va{u}_{d}}
      \EE\Bc{\coutint_{d}(x_{d},\va{u}_d,\va{w}_{d})
      +\bscal{\price_{d+1}}{\dynamics_{d}(x,\va{u}_d,\va{w}_{d})} }
      \LowPlus 
      \inf_{\va{x}_{d+1}} \EE \bc{ - \bscal{\price_{d+1}}{\va{x}_{d+1}}
      + \underline{V}_{d+1}^{\mathrm{i}}(\va{x}_{d+1})}
       \intertext{by substituting~\eqref{tts:eq:intrapbdual}, and by using
      subadditivity of the infimum operation with respect to the Moreau lower addition~$\LowPlus$}
     & =
      \sup_{\price_{d+1} \leq 0} \;
      \inf_{\va{u}_{d},\va{x}_{d+1}}
      \EE\Bc{\coutint_{d}(x_{d},\va{u}_d,\va{w}_{d})
      +\bscal{\price_{d+1}}{\dynamics_{d}(x,\va{u}_d,\va{w}_{d}) -\va{x}_{d+1} }
       + \underline{V}_{d+1}^{\mathrm{i}}(\va{x}_{d+1})}
\tag{as all quantities inside the expectations are in $\OpenIntervalClosed{-\infty}{+\infty}$}
    \\
    & \le
      \inf_{\va{u}_{d},\va{x}_{d+1}} \;
      \sup_{\price_{d+1} \leq 0}
       \EE\Bc{\coutint_{d}(x_{d},\va{u}_d,\va{w}_{d})
      +\bscal{\price_{d+1}}{\dynamics_{d}(x,\va{u}_d,\va{w}_{d}) -\va{x}_{d+1} }
      + \underline{V}_{d+1}^{\mathrm{i}}(\va{x}_{d+1})}
   \\
    & \le
      \inf_{\va{x}_{d+1},\va{u}_{d}} \EE\bc{ \coutint_{d}(x,\va{u}_d,\va{w}_{d})
      + V_{d+1}^{\mathrm{i}}(\va{x}_{d+1})}
      \quad  \text{s.t} \quad
      \dynamics_{d}(x,\va{u}_d,\va{w}_{d}) \geq \va{x}_{d+1}
      \tag{by weak duality} \\
    & = \underline{V}_{d}^{\mathrm{i}}(x)\eqfinp
  \end{align*}
  This ends the proof.
\end{proof}

\subsection{Upper bounds of the Bellman value functions}
\label{Upper_bounds_of_the_value_functions}

We present upper bounds for the Bellman value functions $\ba{V_{d}^{\mathrm{i}}}_{d\in\longset\cup\{\longfinal\}}$
given by Equation~\eqref{tts:bellman-2level}.
They are obtained using a kind of resource decomposition
scheme associated with the dynamic equations, that is, by requiring
that the state at time~$d+1$ be set at a prescribed deterministic
value, so that new constraints have to be added. This is made possible
by the fact that we relax the almost sure target equality
constraint~\eqref{tts:eq:2tsdyn} into the inequality
constraint~\eqref{tts:eq:2tsdynmon}, that is, we enlarge the
admissibility set of the problem, without changing the optimal
value of the problem thanks to the dynamically monotone property
(see Proposition~\ref{pr:main_result_of_Section2}).

We define the function
$\coutint_{d}^{\mathrm{R}}:\XX_{d}\times\XX_{d+1}\to \ClosedIntervalClosed{0}{+\infty}$ by
\begin{subequations}
  \label{tts:eq:intrapbrelaxed}
  \begin{align}
    \coutint_{d}^{\mathrm{R}}(x_{d},r_{d+1}) = \inf_{\va{u}_{d}} \;
    & \EE \; \Bc{\coutint_{d}(x_{d},\va{u}_d,\va{w}_{d})} \eqfinv \\
    \text{s.t.} \quad
    & \dynamics_{d}(x_{d},\va{u}_d,\va{w}_{d}) \ge r_{d+1}
      \label{tts:eq:ineqtarget} \eqfinv \\
    & \sigma(\va U_{d,m}) \subset \sigma(\va{w}_{d,0:m})
      \eqsepv \forall m \in \shortset \eqfinv
  \end{align}
\end{subequations}
where~$\coutint_{d}$ and~$\dynamics_{d}$ are respectively the instantaneous
cost function and the dynamics of Problem~\eqref{tts:eq:2tspb}. Note that the
function $\coutint_{d}^{\mathrm{R}}$ can take the value $+\infty$ since the
constraint~\eqref{tts:eq:ineqtarget} may lead to an empty admissibility set.
Having replaced the equality constraint~\eqref{tts:eq:2tsdyn} by the
  inequality constraint~\eqref{tts:eq:2tsdynmon} in
  Problem~\eqref{tts:eq:2tspbmon} makes possible to have the inequality
  constraint~\eqref{tts:eq:ineqtarget} in the definition of the
  function~$\coutint_{d}^{\mathrm{R}}$.  This last inequality ensures that a
  random variable is almost surely greater or equal to a deterministic quantity,
  a much more easier situation that ensuring the equality between a random
  variable and a deterministic quantity.

\begin{proposition}
  \label{tts:prop:relaxedintraineq}
  Suppose that the white noise Assumption~\ref{tts:hyp:indep}
  and the Nonincreasing Bellman value functions
  Assumption~\ref{tts:hyp:nonincreasing} hold true. Consider the
  sequence~$\ba{\overline{V}_{d}^{\mathrm{R}}}_{d\in\longset\cup\{\longfinal\}}$
   of Bellman value functions
  defined inductively by $\overline{V}_{\longfinal}^{\mathrm{R}} = \coutfin$
  and for all $d\in\longset$ and for all $x\in\XX_d$ by
  \begin{equation}
    \label{tts:eq:bellman-relaxed-det}
    \overline{V}_{d}^{\mathrm{R}}(x) =
    \inf_{r_{d+1} \in\XX_{d+1}} \bp{ \coutint_{d}^{\mathrm{R}}(x,r_{d+1})
      + \overline{V}_{d+1}^{\mathrm{R}}(r_{d+1})} \eqfinp
  \end{equation}
  Then, the Bellman value
  functions~$\ba{\overline{V}_{d}^{\mathrm{R}}}_{\longset\cup\{\longfinal\}}$
  given by Equation~\eqref{tts:eq:bellman-relaxed-det}
  are upper bounds of the Bellman value functions
  $\ba{V_{d}^{\mathrm{i}}}_{d\in\longset\cup\{\longfinal\}}$
  given by Equation~\eqref{tts:bellman-2level}, that is,
  \begin{equation}
    V_{d}^{\mathrm{i}} \le \overline{V}_{d}^{\mathrm{R}}
    \eqsepv \forall d\in\longset\cup\{\longfinal\} \eqfinp
  \end{equation}
\end{proposition}

\begin{proof}
  The proof is done by backward induction. We first have that
  $\overline{V}_{D+1}^{\mathrm{R}} = \coutfin = V_{D+1}^{\mathrm{i}}$.
  Now, consider $d\in \longset$ and assume that
  $V_{d+1}^{\mathrm{i}} \le \overline{V}_{d+1}^{\mathrm{R}}$.
  We successively have
  \begin{align*}
    V_{d}^{\mathrm{i}}(x)
    & = \inf_{\va{x}_{d+1},\va{u}_{d}} \EE \bc{ \coutint_{d}(x,\va{u}_d,\va{w}_{d})
      + V_{d+1}^{\mathrm{i}}(\va{x}_{d+1})} \\
    & \hspace{1cm} \text{s.t} \quad \dynamics_{d}(x,\va{u}_d,\va{w}_{d}) \geq \va{x}_{d+1},
      \text{ and } \eqref{tts:bellman-2level-mes} \text{--} \eqref{tts:bellman-2level-mesX},
      \tag{by~\eqref{tts:bellman-2level}} \\
    & \le \inf_{{r}_{d+1} \in \XX_{d+1}} \inf_{\va{u}_{d}}
      \EE \bc{ \coutint_{d}(x,\va{u}_d,\va{w}_{d})} +
      V_{d+1}^{\mathrm{i}}({r}_{d+1}) \\
    & \hspace{1cm} \text{s.t} \quad \dynamics_{d}(x,\va{u}_d,\va{w}_{d}) \geq {r}_{d+1}
      \text{ and } \eqref{tts:bellman-2level-mes}
      \tag{by considering constant r.v.~$\va{x}_{d+1}$} \\
    & =  \inf_{{r}_{d+1} \in \XX_{d+1}}
      \bp{\coutint_{d}^{\mathrm{R}}(x_{d},r_{d+1}) +
      V_{d+1}^{\mathrm{i}}({r}_{d+1})}
      \tag{by definition~\eqref{tts:eq:intrapbrelaxed}} \\
    & \le  \inf_{{r}_{d+1} \in \XX_{d+1}}
      \bp{\coutint_{d}^{\mathrm{R}}(x_{d},r_{d+1}) +
      \overline{V}_{d+1}^{\mathrm{R}}({r}_{d+1})}
      \tag{by the induction assumption} \\
    & = \overline{V}_{d}^{\mathrm{R}}(x)\eqfinp
      \tag{by definition~\eqref{tts:eq:bellman-relaxed-det}}
  \end{align*}
  This ends the proof.
\end{proof}

\subsection{Periodicity classes and algorithms}
\label{subsec:periodicity}

To compute the upper bound Bellman value functions $\ba{\overline{V}_{d}^{\mathrm{R}}}_{\longset\cup\{\longfinal\}}$
given by Equation~\eqref{tts:eq:bellman-relaxed-det} (respectively
the lower bound Bellman value functions $\ba{\underline{V}_{d}^{\mathrm{P}}}_{d\in\longset\cup\{\longfinal\}}$
given by Equation~\eqref{tts:eq:bellman-dualized-det}), we need
to compute the value of the fast scale optimization problems $\coutint_{d}^{\mathrm{R}}(x_{d},r_{d+1})$ defined
by~\eqref{tts:eq:intrapbrelaxed}
(respectively~$\coutint_{d}^{\mathrm{P}}(x_{d},\price_{d+1})$ defined
by~\eqref{tts:eq:intrapbdual}), for all~$d\in\longset\cup\{\longfinal\}$
and for all couples~$(x_{d},r_{d+1})\in\XX_d\times\XX_{d+1}$ (respectively
for all couples~$(x_{d},\price_{d+1})\in\XX_d\times\pricespace_{d+1}$
with~$\price_{d+1} \leq 0$).  The computational cost can be significant
as we need to solve a stochastic optimization problem for every
couple~$(x_d,r_{d+1})\in\XX_d\times\XX_{d+1}$, for every
couple~$(x_d,\price_{d+1})\in\XX_d\times\pricespace_{d+1}$ and
for every~$d$ in~$\longset$. We present a simplification exploiting
a possible periodicity at the fast time scale.

\begin{proposition}
  \label{tts:prop:periodicity}
  Let~$\mathbb{I} \subset \longset$ be a nonempty subset.
  Assume that there exists two sets~$\XX_{\mathbb{I}}$
  and~$\UU_{\mathbb{I}}$ such that for any~$d \in \mathbb{I}$, $\XX_{d} = \XX_{\mathbb{I}}$
  and $\UU_{d} = \UU_{\mathbb{I}}$. Assume, moreover,
  that there exists two mappings
  $\coutint_{\mathbb{I}}$ and $\dynamics_{\mathbb{I}}$ such that
  for any~$d \in \mathbb{I}$,
  $\coutint_{d} = \coutint_{\mathbb{I}}$ and~$\dynamics_{d} = \dynamics_{\mathbb{I}}$.
  Finally assume that the random variables~$\{ \va{w}_{d}\}_{d \in \mathbb{I}}$
  are independent and identically distributed. Then, there exists two
  functions~$\coutint_{\mathbb{I}}^{\mathrm{R}}$ and~$\coutint_{\mathbb{I}}^{\mathrm{P}}$
  such that the
  functions~$\coutint_{d}^{\mathrm{R}}$ defined by~\eqref{tts:eq:intrapbrelaxed}
  and~$\coutint_{d}^{\mathrm{P}}$ defined by~\eqref{tts:eq:intrapbdual}
  satisfy
  \begin{equation}
    \label{tts:eq:dailyeq}
    \coutint_{d}^{\mathrm{R}} = \coutint_{\mathbb{I}}^{\mathrm{R}}
    \quad \text{and} \quad
    \coutint_{d}^{\mathrm{P}} = \coutint_{\mathbb{I}}^{\mathrm{P}}
    \eqsepv \forall d \in \mathbb{I}
    \eqfinp
  \end{equation}
\end{proposition}

\begin{proof}
  The proof is an immediate consequence of the assumptions made on the
  functions~$\coutint_{d}$ and~$\dynamics_{d}$ and on~$\va{w}_{d}$.
\end{proof}

The nonempty set~$\mathbb{I}$ in Proposition~\ref{tts:prop:periodicity}
is called a \emph{periodicity class}. We denote by~$I$ the number of
periodicity classes of Problem~\eqref{tts:eq:2tspb} and
by~$(\mathbb{I}_1,\dots,\mathbb{I}_{I})$ the periodicity classes (possibly
reduced to singletons), that is,
the sets of slow time indices that satisfy~\eqref{tts:eq:dailyeq}.

\begin{remark}
  A periodicity property often appears in long term energy management
  problems with renewable energies, due to yearly seasonality of natural
  processes such as solar production. In these cases~$I<D+1$ and
  it is enough to solve only~$I$ problems at the fast time scale.

  When there is no periodicity, $I = D+1$ and the periodicity
  classes are singletons. In this case all the fast time scale
  problems have to be computed.
\end{remark}

The ways to obtain the upper bound and lower bound Bellman value
functions are presented in Algorithm~\ref{tts:alg:2tssdp}
and~Algorithm~\ref{tts:alg:2tssdpdual}. They allow to efficiently
compute fast time scale problems and upper and lower bounds for the
Bellman value functions using resource and price decompositions.

\begin{algorithm}[H]
  \KwData{Periodicity classes~$(\mathbb{I}_1, \dots, \mathbb{I}_{I})$}
  \KwResult{Upper bound Bellman value functions
  $\ba{\overline{V}_{d}^{\mathrm{R}}}_{d\in\longset\cup\{\longfinal\}}$}
  \textbf{Initialization}: $\overline{V}_{\longfinal}^{\mathrm{R}} = \coutfin$ \\
  \For{$i = 1,\dots,I$}{
    Choose a~$d_{i} \in \mathbb{I}_{i}$ \\
    \For{$(x,r) \in \XX_{d_{i}} \times \XX_{d_{i}+1}$}{
      Compute~$\coutint_{\mathbb{I}_{i}}^{\mathrm{R}}(x,r)=\coutint_{d_{i}}^{\mathrm{R}}(x,r)$
      and set~$\coutint_{d}^{\mathrm{R}}(x,r)=\coutint_{\mathbb{I}_{i}}^{\mathrm{R}}(x,r)$
      for all~$d\in\mathbb{I}_{i}$}
  }
  \For{$d = D,D-1,\dots,0$}{
    \For{$x_d \in \XX_{d}$}{
      Solve~$\overline{V}_{d}^{\mathrm{R}}(x_d) = \inf_{r_{d+1}}
      \coutint_{d}^{\mathrm{R}}(x_{d},r_{d+1})
      + \overline{V}_{d+1}^{\mathrm{R}}(r_{d+1})$
    }
  }
  \caption{Two-time-scale \DP\ with deterministic resources
           and periodicity classes}
  \label{tts:alg:2tssdp}
\end{algorithm}

\begin{algorithm}[H]
  \KwData{Periodicity classes~$(\mathbb{I}_1,\dots,\mathbb{I}_{I})$}
  \KwResult{Lover bound Bellman value functions
  $\ba{\underline{V}_{d}^{\mathrm{P}}}_{d\in\longset\cup\{\longfinal\}}$}
  Initialization: $\underline{V}_{D+1}^{\mathrm{P}} = \coutfin$ \\
  \For{$i = 1,\dots,I$}{
    Choose a~$d_{i} \in \mathbb{I}_{i}$ \\
    \For{$(x,\price) \in \XX_{d_{i}} \times \pricespace_{d_{i}+1} , \; \price \leq 0$}{
      Compute~$\coutint_{\mathbb{I}_{i}}^{\mathrm{P}}(x,\price)=\coutint_{d_{i}}^{\mathrm{P}}(x,\price)$
      and set~$\coutint_{d}^{\mathrm{P}}(x,\price)=\coutint_{\mathbb{I}_{i}}^{\mathrm{P}}(x,\price)$ for all~$d\in\mathbb{I}_{i}$}
  }
  \For{$d = D,D-1,\dots,0$}{
    \For{$x_d \in \XX_{d}$}{
      Solve~$\underline{V}_{d}^{\mathrm{P}}(x) = \sup_{\price_{d+1} \leq 0}
      \coutint_{d}^{\mathrm{P}}(x,\price_{d+1})
      - \LFM{\bp{\underline{V}_{d+1}^{\mathrm{P}}}}(\price_{d+1})$
    }
  }
  \caption{Two-time-scale \DP\ with deterministic prices
           and periodicity classes}
  \label{tts:alg:2tssdpdual}
\end{algorithm}

\begin{remark}
  The interest of Algorithm \ref{tts:alg:2tssdp}
  (resp. \ref{tts:alg:2tssdpdual}) is that we can solve
  the~$I$ fast time scale problems~\eqref{tts:eq:intrapbrelaxed}
  (resp. \eqref{tts:eq:intrapbdual}) in parallel and then distribute
  the numerical solving of these problems across slow time steps.
  Moreover, we can theoretically apply any stochastic optimization method
  to solve the fast time scale problems. Without stagewise independence
  assumption at the fast time scale, we may use \SP\ techniques (for example
  scenario trees and progressive hedging~\cite{Rockafellar-Wets:1991})
  to solve the fast time scale problems.
  With stagewise independence assumption, we may apply \SDP\ or \SDDP.
\end{remark}

\subsection{Gathering the results to compute policies}
\label{sec:computing_policies}

We assume that we have at disposal Bellman value
functions~$\{\widetilde V_d\}_{d\in\longset\cup\longfinal}$
obtained either by resource decomposition
($\widetilde V_d = \overline{V}_{d}^{\mathrm{R}}$),
or by price decomposition ($\widetilde V_d =
\underline{V}_{d}^{\mathrm{P}}$).
The computation of the~$\widetilde V_d$'s, that
is, the computation of the~$\overline{V}_{d}^{\mathrm{R}}$'s
or~$\underline{V}_{d}^{\mathrm{P}}$'s, constitutes
the \emph{offline} part of the algorithms, as described
in algorithms~\ref{tts:alg:2tssdp} and~\ref{tts:alg:2tssdpdual}.
We recall that, under Assumption~\ref{tts:hyp:indep}
and~\ref{tts:hyp:nonincreasing}, and
thanks to Propositions~\ref{prop:dailymonotone},
\ref{tts:prop:dualintraineq} and~\ref{tts:prop:relaxedintraineq},
we have the inequalities
\begin{equation*}
  \underline{V}_{d}^{\mathrm{P}}
  \leq V_{d}^{\mathrm{e}} \leq \overline{V}_{d}^{\mathrm{R}}
  \eqsepv \forall d\in\longset\cup\{\longfinal\} \eqfinp
\end{equation*}
Then, for a given slow time step~$d\in\longset$ and a given current
state $x_d \in \XX_d$, we can use~$\widetilde V_{d}$ as an approximation
of the Bellman value function~$V_{d}^{\mathrm{e}}$ in order to state
a new fast time scale problem starting at~$d$ for computing the decisions
to apply at~$d$. This constitutes the \emph{online} part of the procedure:
\begin{subequations}
  \label{tts:eq:realintrapb}
  \begin{align}
    \va{u}\opt_{d} \in   \argmin_{\va{u}_{d}} \;
    & \EE \Bc{\coutint_{d}(x_{d},\va{u}_{d},\va{w}_{d})
      + \widetilde V_{d+1} \bp{\dynamics_{d}(x_{d},\va{u}_{d},\va{w}_{d})}}
      \eqfinv \\
    \text{s.t.} \quad
    & \sigma(\va U_{d,m}) \subset \sigma(\va{w}_{d,0:m})
      \eqsepv \forall m \in \shortset \eqfinp
  \end{align}
\end{subequations}
This problem can be solved by any method that provides an online
policy as presented in~\cite{tts:bertsekas2005dynamic}.
The presence of a final cost~$\widetilde V_{d}$ ensures that
the effects of decisions made at the fast time scale are taken
into account at the slow time scale.

Nevertheless, it would be time-consuming to produce online policies
using the numerical solving of Problem~\eqref{tts:eq:realintrapb} for every
slow time step of the horizon in simulation. We present in the next
two paragraphs how to obtain two-time-scale policies with
prices or resources in a smaller amount of time.

\subsubsubsection{Obtaining a policy using prices}

In the case where we decompose the problem using deterministic prices,
we possibly solve Problem~\eqref{tts:eq:intrapbdual} for every
couple of initial state and deterministic price
$(x_d,\price_{d+1}) \in \XX_d \times \pricespace_{d+1}$ and
for every $d\in\longset$. In the process, an optimal policy for
every problem has been computed. For~$d\in\longset$ and
for~$(x_d,\price_{d+1}) \in \XX_d \times \pricespace_{d+1}$,
we call~$\pi_{d}^{\mathrm{P}}(x_d,\price_{d+1}): \WW_d \to \UU_{d}$
an optimal policy for Problem~\eqref{tts:eq:intrapbdual}, whose value
is~$\coutint_{d}^{\mathrm{P}}(x_d,\price_{d+1})$.

At the beginning of a slow time step~$d$ in a state~$x_d \in \XX_d$,
we compute a price~$\price_{d+1}$ solving the following optimization
problem
\begin{equation}
  \label{tts:eq:simulatingprice}
  \price_{d+1} \in \argmax_{\price \leq 0}
  \Bp{\coutint_{d}^{\mathrm{P}}(x_d, \price) -
    \LFM{\bp{\underline{V}_{d+1}^{\mathrm{P}}}}(\price)} \eqfinp
\end{equation}
Thanks to this deterministic price~$\price_{d+1}$, we apply the corresponding
policy~$\pi_{d}^{\mathrm{P}}(x_d, \price_{d+1})$ to simulate  decisions
and states drawing a scenario~$w_d$ out of the random process~$\va W_d$.
The next state~$x_{d+1}$ at the beginning of the slow time step~$d+1$ is then
$x_{d+1} = \dynamics_{d} \bp{x_d, \pi_{d}^{\mathrm{P}}(x_d, \price_{d+1})(w_d), w_d}$.

\subsubsubsection{Obtaining a policy using resources}

In the case where we decompose the problem using deterministic resources,
we possibly solve Problem~\eqref{tts:eq:intrapbrelaxed} for every
couple of initial state and deterministic resource
$(x_d, x_{d+1}) \in \XX_d \times \XX_{d+1}$ and for every $d \in\longset$.
In the process, an optimal policy for every problem has been computed.
For~$d\in\longset$ and for~$(x_d,x_{d+1}) \in \XX_d \times \XX_{d+1}$,
we call~$\pi_{d}^{\mathrm{R}}(x_d, x_{d+1}): \WW_d \to \UU_{d}$
an optimal policy for Problem~\eqref{tts:eq:intrapbrelaxed}
whose value is~$\coutint_{d}^{\mathrm{R}}(x_d, x_{d+1})$.

At the beginning of a slow time step~$d$ in a state~$x_d \in \XX_d$,
we compute a resource (state)~$x_{d+1}$ solving the following optimization
problem
\begin{equation}
  \label{tts:eq:simulatingresource}
  x_{d+1} \in \argmin_{x \in \XX_{d+1}}
  \Bp{\coutint_{d}^{\mathrm{R}}(x_d, x) + { \overline{V}_{d+1}^{\mathrm{R}}(x)}}
  \eqfinv
\end{equation}
and we apply the corresponding policy~$\pi_{d}^{\mathrm{R}}(x_d, x_{d+1})$
to simulate decisions and states drawing a scenario~$w_d$ out
of~$\va W_d$. The next state~$x_{d+1}$ at the beginning
of the slow time step~$d+1$ is then
$x_{d+1}=\dynamics_d \bp{x_d,\pi_{d}^{\mathrm{R}}(x_d, x_{d+1})(w_d),w_d}$.

\section{Case study}
\label{tts:sec:experiments}

In this section, we apply the previous theoretical results
to a long term aging and battery renewal management problem.
In~\S\ref{tts:ssec:experiments-formulation}, we formulate
the problem. In~\S\ref{Simplifying_the_intraday_problems},
we simplify the intraday problems.
In~\S\ref{tts:ssec:experiments-setup},
we describe the data used for the numerical experiments.
Finally, in~\S\ref{Numerical_experiments}, we sketch how
to apply resource and price decomposition algorithms,
and we compare the results given by each of these methods.

\subsection{Problem formulation}
\label{tts:ssec:experiments-formulation}

As explained in~\S\ref{tts:ssect-motivation}, we are dealing with
the energy storage management problem of a battery over a very long
term ($20$~years). We adopt the notations defined in~\S\ref{ssect:tts-notations}.
The total number of slow time steps (days) in the time horizon is denoted
by~$D+1$ ($D= 20 \times 365 = 7300$),
and each slow time interval~$\ClosedIntervalOpen{d}{d+1}$ consists
of~$M+1$ fast time steps (half hour, hence $M = 24 \times 2 = 48$).

\begin{subequations}
  \label{tts:dynamics-detailed-fast}
  At the fast time scale, the system control is the energy~$\va{u}_{d,m}$
  transferred in the battery. We denote the charge of the battery
  by~$\va{u}_{d,m}^+ = \max \{0,\va{u}_{d,m}\}$, and the discharge
  of the battery  by~$\va{u}_{d,m}^- = \max \{0,-\va{u}_{d,m}\}$.
  For all time\footnote{There
  is here a slight difference with the notations presented
  in~\S\ref{ssect:tts-notations}: we have added a new time step
  $(d,M{+}1)$ at the end of day~$d$ in order to apply the last
  fast control of day~$d$ and the slow control of day~$d{+}1$
  at distinct time steps, hence the introduction of a fictitious
  time step --- denoted by~$(d,M{+}1)$ --- between~$(d,M1)$ and~$(d{+}1,0)$
  (see Equation~\eqref{tts:dynamics-detailed-last} and comments above).}~$(d,m)\in\longset \times \timeset{0}{M{+}1}$,
  the state of the battery consists of
  \begin{itemize}
  \item
    the amount of energy in the battery~$\va{s}_{d,m}$
    (state of charge), whose dynamics is given by the simple storage dynamics equation
    \begin{equation}
      \label{tts:dynamics-detailed-fast-S}
      \va{s}_{d,m+1} = \va{s}_{d,m} +
      \rho^{\mathrm{c}} \va{u}_{d,m}^{+} - \rho^{\mathrm{d}} \va{u}_{d,m}^{-} \eqfinv
    \end{equation}
    where~$\rho^{\mathrm{c}}$ and~$\rho^{\mathrm{d}}$ are the charge
    and discharge coefficients of the battery,
  \item
    the amount of remaining exchangeable energy~$\va{h}_{d,m}$
    (health of the battery), with
    \begin{equation}
      \label{tts:dynamics-detailed-fast-H}
      \va H_{d,m+1} = \va H_{d,m} -
      \va U_{d,m}^+ - \va U_{d,m}^- \eqfinv
    \end{equation}
    so that health decreases with variations in energy exchanges,
  \item
    the capacity~$\va{c}_{d,m}$ of the battery
    (assumed to be constant at the fast time scale)
    \begin{equation}
      \label{tts:dynamics-detailed-fast-C}
      \va{c}_{d,m+1} = \va{c}_{d,m} \eqfinp
    \end{equation}
  \end{itemize}
  These equations at the fast time scale are gathered as
  \begin{equation}
    \label{tts:dynamics-regrouped-fast}
    (\va{s}_{d,m+1},\va{h}_{d,m+1},\va{c}_{d,m+1}) =
    \varphi\bp{\va{s}_{d,m},\va{h}_{d,m},\va{c}_{d,m},\va{u}_{d,m}} \eqfinp
  \end{equation}
\end{subequations}
\begin{subequations}
  \label{tts:dynamics-detailed-last}
  At the slow time scale, that is, for each slow time step~$d$,
  there exists another control~$\va{r}_{d}$ modeling the possible
  renewal of the battery at the end of the slow time step. To take
  it into account, we add a fictitious time step~$(d,M{+}1)$ between~$(d,M)$
  and~$(d{+}1,0)$. The dynamics of the battery for this specific time
  step are:
  \begin{align}
    & \va{s}_{d+1,0} =
      \begin{cases}
        0              & \text{ if } \va{r}_{d} > 0 \eqfinv \\
        \va{s}_{d,M+1} & \text{ otherwise} \eqfinv
      \end{cases} \label{tts:dynamics-detailed-last-S}
    \intertext{meaning that, when renewed, a new battery is empty,}
    & \va{h}_{d+1,0} =
      \begin{cases}
        \fhealth(\va{r}_{d}) \va{r}_{d} & \text{ if } \va{r}_{d} > 0 \eqfinv \\
        \va{h}_{d,M+1}             & \text{ otherwise} \eqfinv
      \end{cases} \label{tts:dynamics-detailed-last-H}
     \intertext{meaning that, when renewed, the health of
     a battery  is the product of the new battery capacity~$\va{r}_{d}$
     by an integer-valued function~$\fhealth : \RR_+ \to \NN$ estimated at~$\va{r}_{d}$,}
    & \va{c}_{d+1,0} =
      \begin{cases}
        \va{r}_{d}     & \text{ if } \va{r}_{d} > 0 \eqfinv \\
        \va{c}_{d,M+1} & \text{ otherwise} \eqfinv
      \end{cases} \label{tts:dynamics-detailed-last-C}
  \end{align}
  corresponding to the renewal of the battery.
  These equations are gathered as:
  \begin{equation}
    \label{tts:dynamics-regrouped-last}
    (\va{s}_{d+1,0},\va{h}_{d+1,0},\va{c}_{d+1,0}) =
    \psi\bp{\va{s}_{d,M+1},\va{h}_{d,M+1},\va{c}_{d,M+1},\va{r}_{d}} \eqfinp
  \end{equation}
\end{subequations}
We assume that the initial state of the battery is known, that is,
$(\va{s}_{0,0},\va{h}_{0,0},\va{c}_{0,0})= (s_{0},h_{0},c_{0})$.

\begin{subequations}
  All the control variables are subject to bound constraints
  \begin{align}
    & \va U_{d,m} \in \nc{\underline{U},\overline{U}} \quad , \quad
      \va{r}_{d} \in \nc{0,\overline{R}} \eqfinv
      \intertext{(with~$\underline{U}<0$ and~$\overline{U}>0$),
      as well as the state variables:}
    & \va S_{d,m} \in \nc{0,\FractionOfTheCapacity \, \va{c}_{d,m}} \quad , \quad
      \va H_{d,m} \in \bc{0,\fhealth(\va{c}_{d,m}) \, \va{c}_{d,m}} \quad , \quad
      \va c_{d,m} \in \nc{0,\overline{R}} \eqfinp
  \end{align}
  Note that the amount of remaining exchangeable energy~$\va{h}_{d,m}$
  has to be positive for the battery to operate, and that the upper
  bound on the state of charge~$\va{S}_{d,m}$ is a fraction~$\FractionOfTheCapacity$
  of the capacity~$\va{c}_{d,m}$.
  \label{tts:bounds-detailed}
\end{subequations}

At each fast time step~$(d,m)$, a local renewable energy production
unit produces energy and a local demand consumes energy: we denote
by~$\va{d}_{d,m}$ the net demand (consumption minus production) and we
suppose that it is an exogenous random variable.
The excess energy consumption
$\max(0,\va{d}_{d,m} + \va{u}_{d,m}^+ - \va{u}_{d,m}^-)$
is paid at a given price~$\pi^e_{d,m}$ (whereas excess energy production
is assumed to be wasted). The price~$\va P^b_d$ of a new battery is also
supposed to be random,
so that the total cost during the slow time step~$d$ is
\begin{equation}
  \label{tts:cost-detailed}
  \sum_{m=0}^{M} \pi^e_{d,m}
  \max(0,\va{d}_{d,m} + \va{u}_{d,m}^+ - \va{u}_{d,m}^-)
  + \va{p}^b_d \va R_d \eqfinp
\end{equation}
The value of battery at the end of the optimization horizon is
represented by a cost function~$K$ depending on the state
of the battery. Then, the objective function to be minimized
is the expected sum over the optimization horizon of the discounted
daily costs (discount factor~$\gamma$), plus the final cost~$K$.
Finally, the optimization problem under consideration is:
\begin{subequations}
  \label{eq:2tsmotivlong}
  \begin{align}
    \inf_{\na{\va{u}_{d,0:M},\va{r}_{d}}_{d\in\timeset{0}{D}}} \;
    & \espe \Bigg[
      \sum_{d=0}^D \gamma^{d}
      \bgp{ \sum_{m=0}^{M} \pi^e_{d,m}
      \max(0,\va{d}_{d,m} + \va{u}_{d,m}^+ - \va{u}_{d,m}^-)
      + \va{p}^{b}_{d} \va R_d } \nonumber \\
    & \qquad\qquad\qquad\qquad\qquad\qquad\qquad +
      \coutfin(\va S_{D+1,0}, \va H_{D+1,0}, \va C_{D+1,0})\Bigg]
      \label{eq:2tsmotivlong-cost} \eqfinv
      \intertext{subject, for all~$(d,m)\in\longset\times\shortset$,
      to state dynamics}
    & (\va{s}_{0,0},\va{h}_{0,0},\va{c}_{0,0})= (s_{0},h_{0},c_{0})
      \label{eq:2tsmotivlong-dyn-ini} \eqfinv \\
    & (\va{s}_{d,m+1},\va{h}_{d,m+1},\va{c}_{d,m+1}) =
      \varphi\bp{\va{s}_{d,m},\va{h}_{d,m},\va{c}_{d,m},\va{u}_{d,m}}
      \label{eq:2tsmotivlong-dyn-fast} \eqfinv \\
    & (\va{s}_{d+1,0},\va{h}_{d+1,0},\va{c}_{d+1,0}) =
      \psi\bp{\va{s}_{d,M+1},\va{h}_{d,M+1},\va{c}_{d,M+1},\va{r}_{d}}
      \label{eq:2tsmotivlong-dyn-last}\eqfinv
      \intertext{
      to bounds constraints}
    & \va S_{d,m} \in \nc{0,\FractionOfTheCapacity \, \va{c}_{d,m}} \eqsepv
      \va H_{d,m} \in \bc{0,\fhealth(\va{c}_{d,m}) \, \va{c}_{d,m}} \eqsepv
      \va c_{d,m} \in \nc{0,\overline{R}}
      \label{eq:2tsmotivlong-constx} \eqfinv \\
    & \va U_{d,m} \in \nc{\underline{U},\overline{U}} \eqsepv
      \va{r}_{d} \in \nc{0,\overline{R}}
      \label{eq:2tsmotivlong-constu} \eqfinv
      \intertext{and to nonanticipativity constraints}
    & \sigma(\va{u}_{d,m}) \subset
      \sigma\bp{\va{d}_{0,0},\dots,\va{d}_{d,m},
      \va{p}^{b}_{0},\dots,\va{p}^{b}_{d-1}}
      \label{eq:2tsmotivlong-nonantu} \eqfinv \\
    & \sigma(\va{r}_{d}) \subset
      \sigma\bp{\va{d}_{0,0},\dots,\va{d}_{d,M},
      \va{p}^{b}_{0},\dots,\va{p}^{b}_{d}}
      \label{eq:2tsmotivlong-nonantr} \eqfinp
  \end{align}
\end{subequations}
\begin{subequations}
We denote by~$\va{u}_{d}$ the vector of decision variables
to be taken during the slow time step~$d$, that is,
\begin{equation}
\va{u}_{d}=\bp{\na{\va{u}_{d,m}}_{m\in\shortset},\va{r}_{d}}
\eqfinp
\end{equation}
We also denote by~$\va{w}_{d}$ the vector of noise variables
occurring during the slow time step~$d$:
\begin{equation}
\va{w}_{d}=\bp{\na{\va{d}_{d,m}}_{m\in\shortset},\va{p}^{b}_{d}}
\eqfinv
\end{equation}
and by~$\va{x}_{d}$ the vector of state variables at the beginning
of the slow time step~$d$:
\begin{equation}
\va{x}_{d}=\np{\va{s}_{d,0},\va{h}_{d,0},\va{c}_{d,0}} \eqfinp
\end{equation}
\end{subequations}
Problem~\eqref{eq:2tsmotivlong} is amenable to the
form~\eqref{tts:eq:2tspb} and thus fits the framework
developed in \S\ref{tts:sec:bounds} for two-time-scale
optimization problems.
\begin{itemize}
\item
  In the expression of~$(\va{s}_{d+1,0},\va{h}_{d+1,0},\va{c}_{d+1,0})$
  given by~\eqref{eq:2tsmotivlong-dyn-last}, replacing the
  variable~$\va{s}_{d,m}$ recursively from~$m=M+1$ to~$m=1$ by using~\eqref{eq:2tsmotivlong-dyn-fast}, one obtains a slow time
  scale dynamics of the form~\eqref{tts:eq:2tsdyn}:
  \begin{equation}
    \label{tts:dynamics-regrouped-slow}
    (\va{s}_{d+1,0},\va{h}_{d+1,0},\va{c}_{d+1,0}) =
    \dynamics_{d}(\va{s}_{d,0},\va{h}_{d,0},\va{c}_{d,0},
    \va{u}_{d,0:M},\va{r}_{d}) \eqfinv
  \end{equation}
\item The cost function of slow time step~$d$ in \eqref{tts:cost-detailed}
  is obviously a function depending on~$\va{u}_{d}$ and~$\va{w}_{d}$.
  The bound constraints on the control~\eqref{eq:2tsmotivlong-constu}
  (resp. the bound constraints on the state~\eqref{eq:2tsmotivlong-constx})
  only depend on~$\va{u}_{d}$ (resp. on~$(\va X_d,\va U_d,\va W_{d})$:
  indeed, in the same way we obtained Equation~\eqref{tts:dynamics-regrouped-slow},
  replacing in the right-hand side of~\eqref{eq:2tsmotivlong-dyn-fast}
  the state variable~$\va{s}_{d,m'}$ recursively from~$m'=m$ to~$m'=0$
  by using~\eqref{eq:2tsmotivlong-dyn-fast}, we obtain that the state
  $(\va{s}_{d,m+1},\va{h}_{d,m+1},\va{c}_{d,m+1})$ is a function of
  $(\va{s}_{d,0},\va{h}_{d,0},\va{c}_{d,0},\va{u}_{d,0:m})$
  for all~$m\in\timeset{0}{M}$.
  These constraints are incorporated in the cost of slow time
  step~$d$ (see Remark~\ref{rm:constraints}) makes it a extended real-valued
  function of the form $\coutint_{d}(\va X_d,\va U_d,\va W_{d})$
  as in~\eqref{tts:eq:2tscost}.
  The final cost~$K$ is, by definition, a function of~$\va{x}_{D+1}$.
\item Since~$\va{r}_{d}$ (resp.~$\va{p}_{d}^{b}$) represents a control
  (resp.~a noise) at the fictitious time step between~$(d,M+1)$
  and~$(d+1,0)$, the nonanticipativy constraints
  \eqref{eq:2tsmotivlong-nonantu} -- \eqref{eq:2tsmotivlong-nonantr}
  are of the form~\eqref{tts:eq:2tsnonant}.
\end{itemize}
We suppose that Assumption~\ref{tts:hyp:indep} is fulfilled, that is,
the net demands~$\na{\va{d}_{d,m}}_{m\in\shortset}$ and the battery
prices~$\va{p}^{b}_{d}$ are independent between slow time steps
(but can be correlated inside a slow time step).
By Proposition~\ref{prop:dailyBellmanPrinciple}, the solution
of Problem~\eqref{eq:2tsmotivlong} can be computed by \DP\ at the slow
time scale.

It is easy to figure out that the daily Bellman value functions computed
by \DP\ are nonincreasing in the state of charge~$s$ and the state
of health~$h$ because it is always preferable to have a full and
healthy battery. We prove in Appendix~\ref{tts:sec:studybatproblem}
that Assumption~\ref{tts:hyp:nonincreasing} holds true for battery
management problems, so that all the results obtained
in \S\ref{tts:sec:bounds} can be applied.

\subsection{Simplifying the intraday problems}
\label{Simplifying_the_intraday_problems}

We turn now to the computation of the
functions~$\coutint_{d}^{\mathrm{P}}$ in~\eqref{tts:eq:intrapbdual}
and~$\coutint_{d}^{\mathrm{R}}$ in~\eqref{tts:eq:intrapbrelaxed},
that we call \emph{intraday functions} in this case study.
As explained in~\S\ref{Lower_bounds_of_the_value_functions},
(resp.~\S\ref{Upper_bounds_of_the_value_functions}),
the intraday functions
$\coutint_{d}^{\mathrm{P}}$ (resp.~$\coutint_{d}^{\mathrm{R}}$)
depend on the pair~$(x_{d},\price_{d+1})$ (resp.~$(x_{d},r_{d+1})$),
namely the $6$-uple~$(s_{d},h_{d},c_{d},p^{s}_{d+1},p^{h}_{d+1},p^{c}_{d+1})$
(resp.~$(s_{d},h_{d},c_{d},s_{d+1},h_{d+1},c_{d+1})$) in the case
study under consideration. We use here some characteristics
of the problem and we make some approximations to alleviate
the computation of these intraday functions.

\subsubsection{Intraday problem associated with resource decomposition}
\label{ssect:simpl-ressource}

As explained in \S\ref{Upper_bounds_of_the_value_functions},
the aim of the resource decomposition algorithm is to compute,
for all slow time steps~$d\in\longset\cup\{\longfinal\}$,
upper bounds~$\overline{V}_{d}^{\mathrm{R}}$ of the Bellman value
functions associated with Problem~\eqref{eq:2tsmotivlong},
which can be put in the form of  Problem~\eqref{tts:eq:2tspb}.
These upper bounds are obtained by solving a collection
of intraday problems such as~\eqref{tts:eq:intrapbrelaxed}
for each slow time step~$d\in\longset$, and then by solving
the Bellman recursion~\eqref{tts:eq:bellman-relaxed-det}.
The intraday problems have a priori to be solved
for every $6$-tuple $(s_{d},h_{d},c_{d},s_{d+1},h_{d+1},c_{d+1})$,
that is, the state~$(s_{d},h_{d},c_{d})$ at the beginning of the
slow time step and the resource target~$(s_{d+1},h_{d+1},c_{d+1})$
at the end of the slow time step. This extremely computationally
demanding task is greatly simplified thanks to the following
considerations.

\subsubsubsection{Resource intraday function reduction}

Since the capacity component~$\va{c}_{d,m}$ of the state can only change
at the end of a slow time step, it is possible to take the capacity dynamics~$\va{c}_{d,m}$, the capacity control~$\va{r}_{d}$ and the
associated bound constraint, and the cost term $\va{p}_{d}^{b} \va{r}_{d}$
out of the intraday problem and to take them into account in the Bellman
recursion. To achieve that, resource decomposition does not
deal with the dynamics equation~\eqref{tts:dynamics-regrouped-last},
but with Equation~\eqref{tts:dynamics-regrouped-fast} for~$m=M$.
We introduce the two resources~$s_{d+1}$ and~$h_{d+1}$ for the state
of charge and the health of the battery\footnote{We do not associate a
  resource variable~$c_{d+1}$ with the capacity of the battery since this
component of the state is taken into account in the Bellman recursion.}.
Then, the intraday problem~\eqref{tts:eq:intrapbrelaxed} becomes
\begin{subequations}
  \label{tts:eq:intrapbrelaxed-renewal-modified}
  \begin{align}
    & \coutint_{d}^{\mathrm{R}}(s_{d},h_{d},c_{d},s_{d+1},h_{d+1})
      = \inf_{\va{u}_{d,0:M}}
      \EE \; \Bgc{ \sum_{m=0}^{M}  \pi^e_{d,m}
      \max(0,\va{d}_{d,m}+\va{u}_{d,m}^{+}-\va{u}_{d,m}^{-}) }
      \eqfinv \\
    & \text{s.t. } \;
      (\va{s}_{d,0},\va{h}_{d,0}) = (s_{d},h_{d})
      \eqfinv \label{tts:eq:intrapbrelaxed-renewal-modified-init}
      \intertext{with, for all~$m\in\shortset$,}
    & \quad\quad \va{s}_{d,m+1} = \va{s}_{d,m} +
      \rho^{\mathrm{c}} \va{u}_{d,m}^{+} - \rho^{\mathrm{d}} \va{u}_{d,m}^{-}
      \eqfinv \label{tts:eq:intrapbrelaxed-renewal-modified-S} \\
    & \quad\quad \va{h}_{d,m+1} = \va{h}_{d,m} -
      \va{u}_{d,m}^{+} - \va{u}_{d,m}^{-}
      \eqfinv \label{tts:eq:intrapbrelaxed-renewal-modified-H} \\
    & \quad\quad \va{s}_{d,M+1} \geq s_{d+1} \eqsepv
      \va{h}_{d,M+1} \geq h_{d+1}
      \eqfinv \label{tts:eq:intrapbrelaxed-renewal-modified-target} \\
    & \quad\quad \va{u}_{d,m} \in \nc{\underline{U},\overline{U}}
      \eqfinv \label{tts:eq:intrapbrelaxed-renewal-modified-boundU} \\
    & \quad\quad \va S_{d,m} \in \nc{0,
      \FractionOfTheCapacity \, c_{d}} \eqsepv
      \va H_{d,m} \in \bc{0,\fhealth(c_{d}) \, c_{d}}
      \eqfinv \label{tts:eq:intrapbrelaxed-renewal-modified-boundSH} \\
    & \quad\quad \sigma(\va{u}_{d,m}) \subset
      \sigma\bp{\va{d}_{d,0},\dots,\va{d}_{d,m}}
      \eqfinv \label{tts:eq:intrapbrelaxed-renewal-modified-mesur}
  \end{align}
and is  parameterized by the $5$-tuple
$(s_{d},h_{d},c_{d},s_{d+1},h_{d+1})$.
\end{subequations}
The sequence
$\{\overline{V}_{d}^{\mathrm{R}}\}_{d\in\longset\cup\{\longfinal\}}$
of Bellman value functions is computed by the following recursion:
\begin{subequations}
  \label{tts:eq:bellman-relaxed-det-renewal-modified}
  \begin{multline}
    \overline{V}_{d}^{\mathrm{R}}(s_{d},h_{d},c_{d}) =
    \inf_{s_{d+1},h_{d+1},\va{r}_{d}}
    \EE \Big[ \gamma \bp{ \coutint_{d}^{\mathrm{R}}(s_{d},h_{d},c_{d},s_{d+1},h_{d+1})
      + \va{p}_{d}^{b} \va{r}_{d} }
    \\
    + \overline{V}_{d+1}^{\mathrm{R}}(\va{s}_{d+1,0},\va{h}_{d+1,0},\va{c}_{d+1,0}) \Big]
    \eqfinv
  \end{multline}
  \begin{align}
    \text{s.t.} \quad
    & (\va{s}_{d+1,0},\va{h}_{d+1,0},\va{c}_{d+1,0}) =
      \psi\bp{s_{d+1},h_{d+1},c_{d},\va{r}_{d}}
      \eqfinv \label{tts:eq:bellman-relaxed-det-renewal-modified-X} \\
    & s_{d+1} \in \nc{0,\FractionOfTheCapacity \, c_{d}} \eqsepv
      h_{d+1} \in \bc{0,\fhealth(c_{d}) \, c_{d}}
      \eqfinv \label{tts:eq:bellman-relaxed-det-renewal-modified-boundX} \\
    & \va{r}_{d} \in \nc{0,\overline{R}}
      \eqfinv \label{tts:eq:bellman-relaxed-det-renewal-modified-boundR} \\
    & \sigma(\va{r}_{d}) \subset \sigma\bp{\va{p}^{b}_{d}}
      \eqfinp \label{tts:eq:bellman-relaxed-det-renewal-modified-mesurR}
  \end{align}
\end{subequations}

{Note that,
  in the Bellman recursion~\eqref{tts:eq:bellman-relaxed-det-renewal-modified},
  we can replace the function $\coutint_{d}^{\mathrm{R}}$
  by the function $\tilde{\coutint}_{d}^{\mathrm{R}}$ with
  \begin{align*}
    \tilde{\coutint}_{d}^{\mathrm{R}}\np{s_{d},h_{d},c_{d},s_{d+1},h_{d+1}}
    &= \coutint_{d}^{\mathrm{R}}\np{s_{d},h_{d},c_{d},s_{d+1},h_{d+1}}
      + \delta_{\nc{0,\fhealth(c_{d}) \, c_{d}}}\np{h_{d}}
    \\
    &+ \delta_{h_{d+1} \le h_d}\np{h_{d},h_{d+1}}
    + \delta_{\nc{0,\fhealth(c_{d}) \, c_{d}}}\np{h_{d+1}}
    \eqfinv
  \end{align*}
  where for any subset \( \Uncertain \subset \RR \),
    $\delta_{\Uncertain} : \RR \to \barRR $ denotes the \emph{indicator function} of the
    set~$\Uncertain$:
    \( \delta_{\Uncertain}\np{\uncertain} = 0 \) if \( \uncertain \in \Uncertain \),
    and \( \delta_{\Uncertain}\np{\uncertain} = +\infty \)
    if \( \uncertain \not\in \Uncertain \).
  Indeed, the last term \( \delta_{\nc{0,\fhealth(c_{d}) \, c_{d}}}(h_{d+1}) \) is obtained by moving
  Constraint~\eqref{tts:eq:bellman-relaxed-det-renewal-modified-boundX} to the minimized cost
  $\coutint_{d}^{\mathrm{R}}$ and
  the two other terms can be added as it is easily seen that
  $\coutint_{d}^{\mathrm{R}}(s_{d},h_{d},c_{d},s_{d+1},h_{d+1})=+\infty$ when $h_{d+1} > h_d$ or when
  $h_d \not\in \nc{0,\fhealth(c_{d}) \, c_{d}}$.
  Now, it is straightforward to prove that
  \begin{equation}
    \label{eq:LR-depends-on-health-delta}
    \tilde{\coutint}_{d}^{\mathrm{R}}(s_{d},h_{d},c_{d},s_{d+1},h_{d+1})
    = \tilde{\coutint}_{d}^{\mathrm{R}}(s_{d},h_{d}{-}h_{d+1},c_{d},s_{d+1},0)
    \eqfinp
  \end{equation}
  Indeed, as the health dynamics is linear nonincreasing, any admissible control
  for the Problem~\eqref{tts:eq:intrapbrelaxed-renewal-modified} for the ordered
  pair $(h_{d},h_{d+1})$, with $h_{d+1} \le h_d$ and
  $(h_d, h_{d+1}) \in  \bc{0,\fhealth(c_{d}) \, c_{d}}^2$ is also admissible for the ordered
  pair $(h_d- h_{d+1},0)$ and conversely. Moreover, the resulting cost is the same
  since the cost does not depend on the health variable. We thus obtain
  Equation~\eqref{eq:LR-depends-on-health-delta}.
}
\subsubsubsection{Resource intraday function approximation}

As suggested in~\cite{KautEA12}, we decide to neglect the state
of charge target at the slow time scale. Indeed, the operation
of the battery is daily periodic and such that it is more or less
empty at the beginning (and thus at the end) of a slow time step
(day). It is thus reasonable to assume that the battery is empty
at the beginning and at the end of every slow time step, which
is a pessimistic but rather realistic assumption. Combined with
Equation~\eqref{eq:LR-depends-on-health-delta}, we obtain a new
function~$\widehat{\coutint}_{d}^{\mathrm{R}}$ approximating
the original function~$\coutint_{d}^{\mathrm{R}}$, that is
\begin{equation}
  \label{eq:LR-depends-on-health-delta-on-no-s}
  \widehat{\coutint}_{d}^{\mathrm{R}}(h_{d}{-}h_{d+1},c_{d})
  = \tilde{\coutint}_{d}^{\mathrm{R}}(0,h_{d}{-}h_{d+1},c_{d},0,0)
  \approx \tilde{\coutint}_{d}^{\mathrm{R}}(s_{d},h_{d},c_{d},s_{d+1},h_{d+1})
  \eqfinp
\end{equation}
The approximated intraday function~$\widehat{\coutint}_{d}^{\mathrm{R}}$
now only depends on two variables, which significantly reduces
the time needed to compute it. Then, the sequence
$\{\overline{V}_{d}^{\mathrm{R}}\}_{d\in\longset\cup\{\longfinal\}}$
of Bellman value functions
in~\eqref{tts:eq:bellman-relaxed-det-renewal-modified} is approximated
by the sequence
$\{\widehat{\overline{V}}_{d}^{\mathrm{R}}\}_{d\in\longset\cup\{\longfinal\}}$
given by the following recursion:
\begin{subequations}
  \label{eq:fully-simplified-resource-bellman-equation}
  \begin{align}
  \widehat{\overline{V}}_{d}^{\mathrm{R}}(h_{d},c_{d}) =
  \inf_{h_{d+1},\va{r}_{d}}
    & \EE \Big[ \gamma
      \bp{ \widehat{\coutint}_{d}^{\mathrm{R}}(h_{d}{-}h_{d+1},c_{d})
      + \va{p}_{d}^{b} \va{r}_{d} }
      + \widehat{\overline{V}}_{d+1}^{\mathrm{R}}
        \bp{\psi^{\mathrm{H,C}}\np{h_{d+1},c_{d},\va{r}_{d}}} \Big]
      \eqfinv \\
  \text{s.t.} \quad
    & h_{d+1} \in \bc{0,\fhealth(c_{d}) \, c_{d}} \eqfinv \\
    & \va{r}_{d} \in \nc{0,\overline{R}} \eqfinv \\
    & \sigma(\va{r}_{d}) \subset \sigma\bp{\va{p}^{b}_{d}} \eqfinp
  \end{align}
  where~$\psi^{\mathrm{H,C}}$ is deduced from~$\psi$
  in~\eqref{tts:dynamics-detailed-last} by keeping only
  the last two dynamics~\eqref{tts:dynamics-detailed-last-H}
  and~\eqref{tts:dynamics-detailed-last-C}, which do not depend
  on the state of charge.
\end{subequations}


As explained in~\S\ref{subsec:periodicity}, we consider
in this study~$I$ periodicity classes ($I=4$, that is,
one class for each season of the year), so that the
computation of the resource intraday problem is done only
for~$I$ different days denoted~$d_{1}$, \dots $d_{I}$.
The complexity of the associated resource decomposition algorithm
is sketched in Appendix~\ref{tts:ann-complexity}.

\subsubsection{Intraday problem associated with price decomposition}
\label{ssect:simpl-price}

As explained in \S\ref{Lower_bounds_of_the_value_functions},
the aim of the price decomposition algorithm is to compute,
for all slow time steps~$d\in\longset\cup\{\longfinal\}$,
lower bounds~$\underline{V}_{d}^{\mathrm{P}}$ of the Bellman
value functions associated with Problem~\eqref{eq:2tsmotivlong}.
These lower bounds are obtained by solving a collection
of intraday problems such as~\eqref{tts:eq:intrapbdual}
for each slow time step~$d\in\longset$, and then by solving
the Bellman recursion~\eqref{tts:eq:bellman-dualized-det}.
The intraday problems have a priori to be solved for every
$6$-tuple $(s_{d},h_{d},c_{d},p^{s}_{d+1},p^{h}_{d+1},p^{c}_{d+1})$,
that is, the state~$(s_{d},h_{d},c_{d})$ at the beginning of the
slow time step and the prices~$(p^{s}_{d+1},p^{h}_{d+1},p^{c}_{d+1})$
associated with the dualization of the equality dynamics equations.

\subsubsubsection{Price intraday function reduction}

As in the resource intraday function reduction, it is possible
to take the capacity dynamics, its associated control and bound
constraints as well as the cost term $\va{p}_{d}^{b} \va{r}_{d}$
out of the intraday problem and to take them into account in the
Bellman recursion, so that the intraday problem does not depend
on the price~$p^{c}_{d+1}$ associated with the capacity dynamics.
To achieve that, price decomposition does not deal with
the dynamics equation~\eqref{tts:dynamics-regrouped-last},
but with Equation~\eqref{tts:dynamics-regrouped-fast} for~$m=M$.
But another possible reduction occurs here: from the health
dynamics~\eqref{tts:dynamics-detailed-fast-H} summed over
the fast time steps of day~$d$, we derive the inequality
\begin{equation}
  h_{d} - \va{h}_{d,M+1} - \sum_{m=0}^{M}
  \bp{\va{u}_{d,m}^{+}+\va{u}_{d,m}^{-}} \geq 0 \eqfinp
\end{equation}
Following the framework of~\S\ref{Lower_bounds_of_the_value_functions},
we dualize it by incorporating on the one hand the terms
$\va{u}_{d,m}^{+}+\va{u}_{d,m}^{-}$ in the definition of the price
intraday function for~$m\in\ic{0,M}$, and on the other hand the term
$h_{d} - \va{h}_{d,M+1}$ in the computation of the Bellman functions.
Doing so, the intraday function~$\coutint_{d}^{\mathrm{P}}$ does not
depend anymore on the health~$h_{d}$, and is defined as
\begin{subequations}
  \label{tts:eq:intrapbdual-renewal-modified}
  \begin{align}
    \coutint_{d}^{\mathrm{P}}
    & (s_{d},c_{d},p^{s}_{d+1},p^{h}_{d+1})
      = \inf_{\va{u}_{d,0:M}} \EE \;
      \bigg[ \sum_{m=0}^{M}
            \Big( \pi^e_{d,m}
                 \max(0,\va{d}_{d,m}+\va{u}_{d,m}^{+}-\va{u}_{d,m}^{-})
            \Big.
      \bigg. \nonumber \\
    & \qquad\qquad\qquad\qquad\qquad\qquad\qquad\qquad\qquad
      - p^{h}_{d+1}\bp{\va{u}_{d,m}^{+}+\va{u}_{d,m}^{-}} \Big)
      + p^{s}_{d+1} \va{s}_{d,M+1} \bigg] \eqfinv
      \intertext{subject to, for all~$m\in\shortset$,}
    & \va{s}_{d,0} = s_{d} \eqsepv
      \va{s}_{d,m+1} = \va{s}_{d,m} +
      \rho^{\mathrm{c}} \va{u}_{d,m}^{+} - \rho^{\mathrm{d}} \va{u}_{d,m}^{-} \eqfinv \\
    & \va{u}_{d,m} \in \nc{\underline{U},\overline{U}} \eqsepv
      \va S_{d,m} \in \nc{0, \FractionOfTheCapacity c_{d}} \eqfinv \\
    & \sigma(\va{u}_{d,m}) \subset
      \sigma\bp{\va{d}_{d,0},\dots,\va{d}_{d,m}} \eqfinp
  \end{align}
\end{subequations}
The associated sequence of Bellman value functions
$\{\underline{V}_{d}^{\mathrm{P}}\}_{d\in\longset\cup\{\longfinal\}}$
is computed by the following recursion:
\begin{subequations}
  \label{tts:eq:bellman-dualized-det-renewal-modified}
  \begin{multline}
    \underline{V}_{d}^{\mathrm{P}}(s_{d},h_{d},c_{d}) =
    \sup_{(p^{s}_{d+1},p^{h}_{d+1})\leq 0}
    \bigg( \coutint_{d}^{\mathrm{P}}(s_{d},c_{d},p^{s}_{d+1},p^{h}_{d+1})
    + \inf_{s_{d,M+1},h_{d,M+1}} \; \inf_{\va{r}_{d}} \;
    \EE \Big[ \gamma \va{p}_{d}^{b} \va{r}_{d} \\
    \qquad\qquad
    - p^{s}_{d+1} s_{d,M+1}
    + \bp{ p^{h}_{d+1}\bp{h_{d}-h_{d,M+1}}
      + \underline{V}_{d+1}^{\mathrm{P}}
      (\va{s}_{d+1,0},\va{h}_{d+1,0},\va{c}_{d+1,0}) }
    \Big]
    \bigg)
    \eqfinv
  \end{multline}
  \begin{align}
    \text{subject to} \quad
    & (\va{s}_{d+1,0},\va{h}_{d+1,0},\va{c}_{d+1,0}) =
      \psi\bp{s_{d,M+1},h_{d,M+1},c_{d},\va{r}_{d}}
      \eqfinv \\
    & \va{r}_{d} \in \nc{0,\overline{R}} \eqsepv
      s_{d,M+1} \in \nc{0,\FractionOfTheCapacity c_{d}} \eqsepv
      h_{d,M+1} \in \bc{0,\fhealth(c_{d}) \, c_{d}}
      \eqfinv \\
    & \sigma(\va{r}_{d}) \subset \sigma\bp{\va{p}^{b}_{d}} \eqfinp
  \end{align}
\end{subequations}

\subsubsubsection{Price intraday function approximation}

As in the resource decomposition algorithm, it
is possible to consider that the state of charge of
the battery has no influence at the slow time scale.
Doing so, we obtain a new
function~$\widehat{\coutint}_{d}^{\mathrm{P}}$ approximating
the original function~$\coutint_{d}^{\mathrm{P}}$, that is
\begin{equation}
  \label{eq:LP-depends-on-c-p-no-s}
  \widehat{\coutint}_{d}^{\mathrm{p}}(c_{d},p^{h}_{d+1})
  = \coutint_{d}^{\mathrm{P}}(0,c_{d},0,p^{h}_{d+1})
  \approx \coutint_{d}^{\mathrm{P}}(s_{d},c_{d},p^{s}_{d+1},p^{h}_{d+1})
  \eqfinp
\end{equation}
The approximated price intraday function~$\widehat{\coutint}_{d}^{\mathrm{P}}$
only depends on the $2$-tuple~$(c_{d},p^{h}_{d+1})$, which
significantly reduces the time needed to compute it. Then, the sequence
$\{\underline{V}_{d}^{\mathrm{P}}\}_{d\in\longset\cup\{\longfinal\}}$
of Bellman value functions
in~\eqref{tts:eq:bellman-dualized-det-renewal-modified}
is approximated by the sequence
$\{\widehat{\underline{V}}_{d}^{\mathrm{P}}\}_{d\in\longset\cup\{\longfinal\}}$
given by the following recursion:
\begin{subequations}
  \label{eq:fully-simplified-price-bellman-equation}
  \begin{multline}
    \widehat{\underline{V}}_{d}^{\mathrm{P}}(h_{d},c_{d}) =
    \sup_{p^{h}_{d+1}\leq 0}
    \bigg( \widehat{\coutint}_{d}^{\mathrm{P}}(c_{d},p^{h}_{d+1})
    + \inf_{h_{d,M+1}} \; \inf_{\va{r}_{d}} \;
    \EE \Big[ \gamma \va{p}_{d}^{b} \va{r}_{d} \\
    \qquad\qquad\qquad
    + p^{h}_{d+1}\bp{h_{d}-h_{d,M+1}}
      + \widehat{\underline{V}}_{d+1}^{\mathrm{P}}
      \bp{\psi^{\mathrm{H},\mathrm{C}}\np{h_{d,M+1},c_{d},\va{r}_{d}}}
    \Big]
    \bigg)
    \eqfinv
  \end{multline}
  \begin{align}
    \text{subject to} \quad
    & h_{d,M+1} \in \bc{0,\fhealth(c_{d}) \, c_{d}}
      \eqfinv \\
    & \va{r}_{d} \in \nc{0,\overline{R}}
      \eqfinv \\
    & \sigma(\va{r}_{d}) \subset \sigma\bp{\va{p}^{b}_{d}} \eqfinp
  \end{align}
\end{subequations}
As explained in~\S\ref{subsec:periodicity}, we consider
in this study~$I$ periodicity classes ($I=4$), that is,
one class for each season of the year), so that
the computation of the price intraday problem is done only
for~$I$ different days denoted~$d_{1}$, \dots $d_{I}$.
The complexity of the associated price decomposition algorithm
is sketched in Appendix~\ref{tts:ann-complexity}.

\subsection{Experimental setup}
\label{tts:ssec:experiments-setup}

The data used in the application come from case studies provided
by a Schneider Electric industrial site,
equipped with solar panels and a battery, and submitted to
three sources of randomness --- namely, solar panels production, electrical
demand and prices of batteries per kWh.
We present hereby
the different parameters of the instance under consideration.
\begin{itemize}
  \item Horizon: $20$ years.
  \item Fast time step: $30$ minutes.
  \item Slow time step: $1$ day.
  \item Number of time steps: $350,\!400$ 
    $(= (24 \times 2)  \times (20 \times 365))$.
  \item Battery renewal capacity: between $0$ and $1,\!500$ kWh
  with a increment of~$100$~kWh.
  \item Periodicity class: $4$ classes, one per trimester of the year.
\end{itemize}

\subsubsubsection{Data to model the cost of batteries and electricity}

For the prices of batteries, we obtained a yearly forecast over~$20$ years from
Statista\footnote{https://www.statista.com/statistics/883118/global-lithium-ion-battery-pack-costs/}.
We added a Gaussian noise to generate synthetic random batteries prices scenarios.
We display in Figure~\ref{tts:fig:batcosts} the scenarios we generated.
\begin{figure}[hbtp]
  \centering
  \mbox{\includegraphics[width=0.8\linewidth]{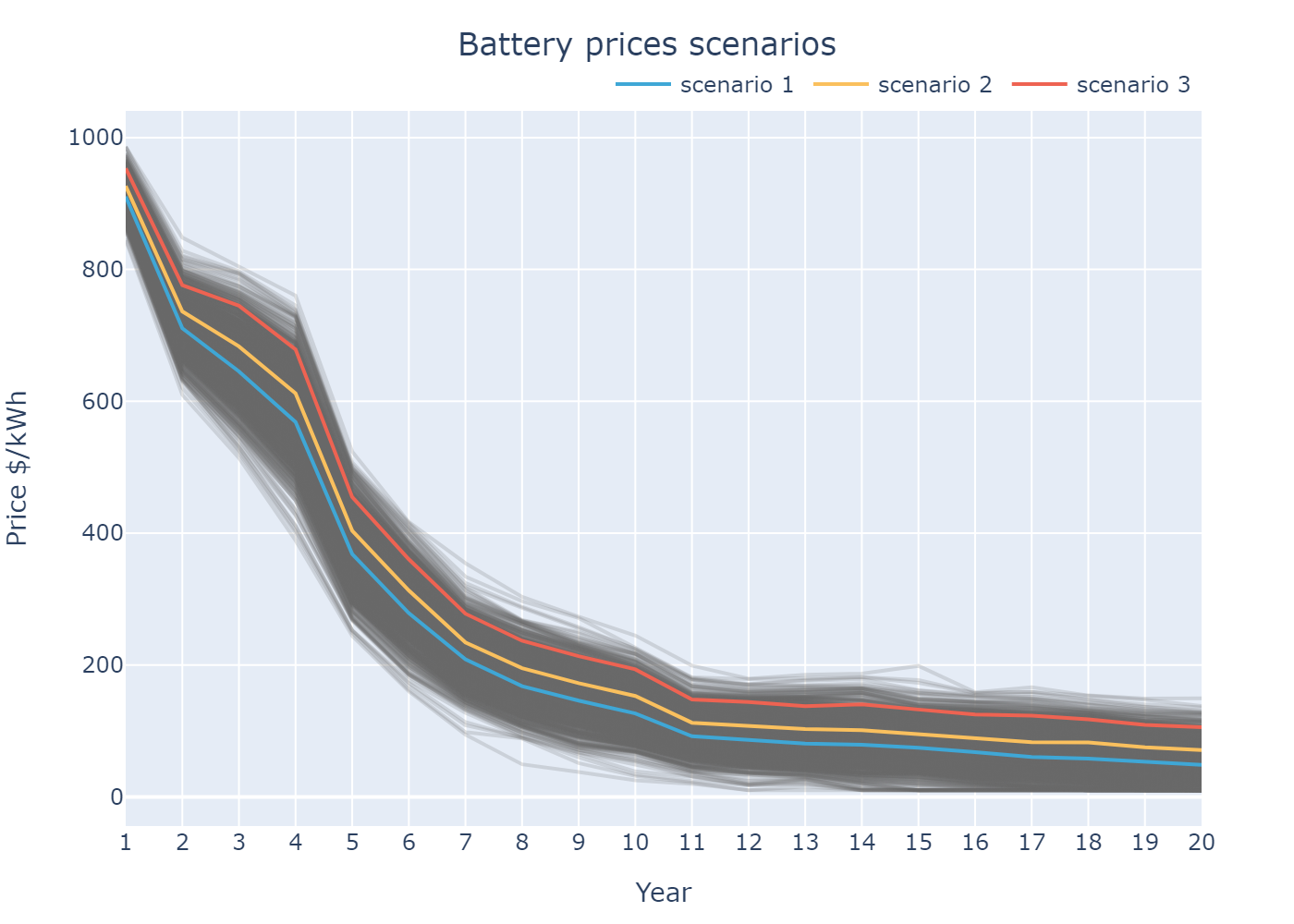}}
  \caption{Scenarios of battery prices over a twenty-year timespan}
  \label{tts:fig:batcosts}
\end{figure}
Three scenarios are highlighted in Figure~\ref{tts:fig:batcosts};
they correspond to the three scenarios we comment in the following
numerical results in~\S\ref{Numerical_experiments}.

For the price of electricity, we chose
a ``time of use'' tariff defined by three rates:
\begin{itemize}
\item an off-peak rate at $0.0255$\$ between 22:00 and 7:00,
\item a shoulder rate at $0.0644$\$ between 7:00 and 17:00,
\item a peak rate at $0.2485$\$ between 17:00 and 22:00.
\end{itemize}

\subsubsubsection{Data to model demand and production}

In order to have a realistic dataset in the model described
in~\S\ref{tts:ssec:experiments-formulation}, we use
the data collected on $70$~anonymized industrial sites
monitored by Schneider Electric. This data set is openly available.\footnote{https://zenodo.org/record/5510400}
We extracted the data of the site numbered~$70$.
For this site, we display in Figure~\ref{tts:fig:netloaddata}
the half hourly distribution of the netload
(demand minus solar production) during one day.
\begin{figure}[hbtp]
  \centering
  \includegraphics[width=0.8\linewidth]{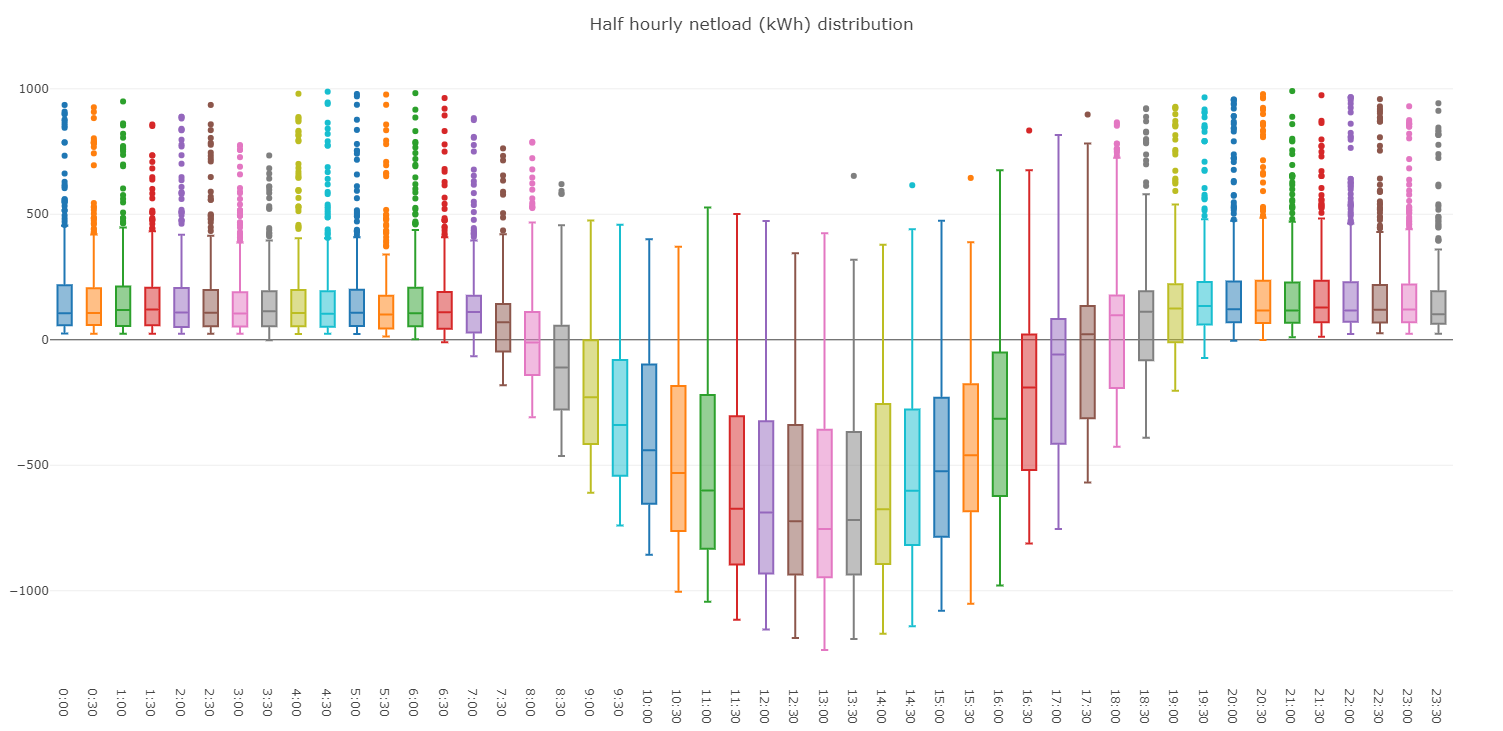}
  \caption{Daily half hourly distribution of netload (kWh)}
  \label{tts:fig:netloaddata}
\end{figure}

\begin{remark}
\label{rem:data_independence}
\textbf{(About the probabilistic independence of the data).}
Both batterie prices and netloads correspond to realistic
data that are given as scenarios, and there is a priori no
independence property for these data.
Of course, it is possible to compute probability laws from
these scenarios: at a given time step~$(d,m)$, collect all
the values~$\va{d}_{d,m}$ (the values~$(\va{d}_{d,M},\va{p}^{b}_{d})$
if~$m=M$) available from the scenarios and build a discrete
probability law from these values. This procedure gives
probability distributions at the half hourly scale,
corresponding to independent random variables.
This way of proceeding will be implemented
in~\S\ref{sect:computation_intraday_functions} to compute the
resource and price intraday functions by dynamic programming.
\end{remark}

\subsection{Numerical experiments}
\label{Numerical_experiments}

The aim of the numerical experiments is to compute and evaluate
policies induced by resource (resp. price) decomposition, that is,
solving an approximation of Problem~\eqref{eq:2tsmotivlong} by
computing the resource intraday functions~$\widehat{\coutint}_{d}^{\mathrm{R}}$
in~\eqref{eq:LR-depends-on-health-delta-on-no-s} and the associated
resource Bellman value functions~$\widehat{\overline{V}}_{d}^{\mathrm{R}}$
in~\eqref{eq:fully-simplified-resource-bellman-equation}
(resp. price intraday functions~$\widehat{\coutint}_{d}^{\mathrm{P}}$
in~\eqref{eq:LP-depends-on-c-p-no-s}
and price Bellman value functions~$\widehat{\underline{V}}_{d}^{\mathrm{P}}$
in~\eqref{eq:fully-simplified-price-bellman-equation})
after all simplifications presented in \S\ref{ssect:simpl-ressource}
(resp. \S\ref{ssect:simpl-price}).

\subsubsection{Computation of the resource and price intraday functions}
\label{sect:computation_intraday_functions}

To compute the approximated resource intraday
functions~$\widehat{\coutint}_{d}^{\mathrm{R}}$ as in Equation~\eqref{eq:LR-depends-on-health-delta-on-no-s}
and the approximated price intraday functions
$\widehat{\coutint}_{d}^{\mathrm{P}}$ in
Equation~\eqref{eq:LP-depends-on-c-p-no-s}, we assume
stagewise independence of the noises at the fast time scale
as detailed in Remark~\ref{rem:data_independence}
and we apply the dynamic programming algorithm.
Indeed, computing the intraday functions using stochastic
programming would be very costly due to the large number
of fast time steps inside a slow time step: for example computing
a price intraday function~\eqref{eq:LP-depends-on-c-p-no-s} would
require forming a scenario tree over~$48$ time steps for every
possible (discretized) value of the pair~$(c_{d},p^{h}_{d+1})$
and solving the associated optimization problem, i.e. a task
too expensive in computation time.
We recall that the intraday functions
are not computed for all possible days in the time horizon,
but only for a day representing each periodicity class.
Here we split the year of the industrial site data into
the four traditional trimesters. For each trimester, we model
the netload at a given half hour of the day by a discrete random
variable with a support of size $10$. The probability distribution
of each discrete random variable is obtained by clustering,
using k-means algorithm, the 
netload realizations in the dataset associated
with the half hour under consideration.

In the case of resource (resp. price) decomposition, we compute
the intraday functions~$\widehat{\coutint}_{d}^{\mathrm{R}}$
(resp. $\widehat{\coutint}_{d}^{\mathrm{P}}$) for every possible
capacity~$c_d$  and every possible exchangeable
energy~$h_{d}-h_{d+1}$ (resp. every possible price~$p^{h}_{d+1}$).
In this study, the possible values of the capacity $c_{d}$
are~$\{0,100\ldots,1500\}$ kWh, whereas the possible values
of the price $p^{h}_{d+1}$ are~$\{0,0.05,0.1\ldots,0.2\}$.

We display in Figure~\ref{tts:fig:intraday1} the resource and price
intraday functions for each season (trimester) of each year.
Resource intraday functions depend on daily exchangeable energy
and capacity, whereas price intraday functions depend on the price
associated with aging and capacity.

\begin{figure}[hbtp]
  \centering
  \includegraphics[width=0.8\linewidth]{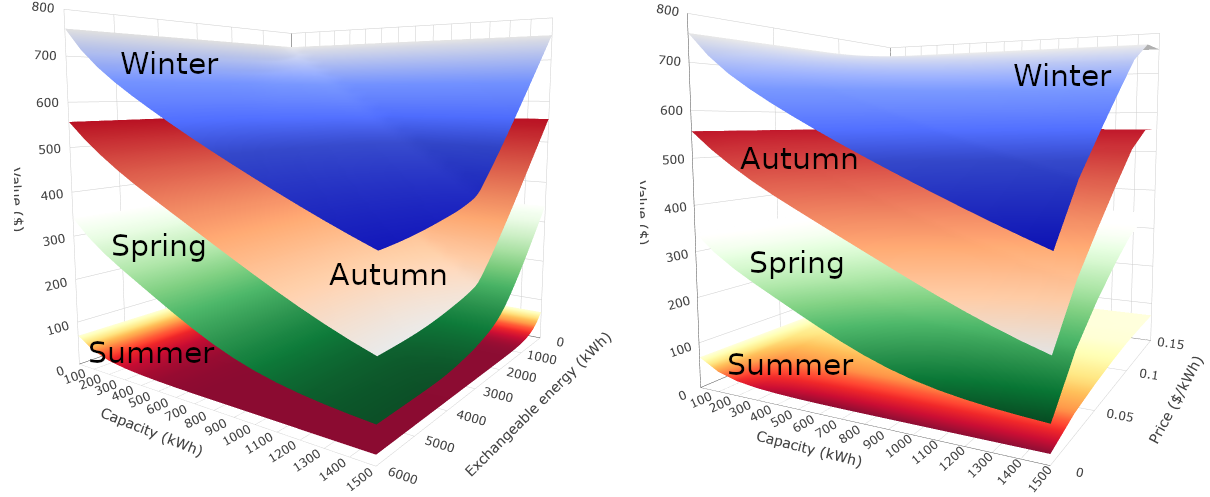}
  \caption{Resource (left) and price (right) intraday functions
           for each trimester}
  \label{tts:fig:intraday1}
\end{figure}

\subsubsection{Computation of the resource and price Bellman value functions}
\label{sect:computing_value_functions}

Once obtained all possible intraday functions
$\widehat{\coutint}_{d}^\mathrm{R}$ and
$\widehat{\coutint}_{d}^\mathrm{P}$, the Bellman
value functions~$\widehat{\overline{V}}_{d}^{\mathrm{R}}$
and~$\widehat{\underline{V}}_{d}^{\mathrm{P}}$
are respectively computed by the Bellman
recursions~\eqref{eq:fully-simplified-resource-bellman-equation}
and~\eqref{eq:fully-simplified-price-bellman-equation},
for~$d\in\longset$.
We display in Figure~\ref{tts:fig:daily_value_function_1}
the resource and price Bellman value functions obtained for the first
day of the time horizon.
\begin{figure}[hbtp]
  \centering
  \mbox{\includegraphics[width=0.7\linewidth]{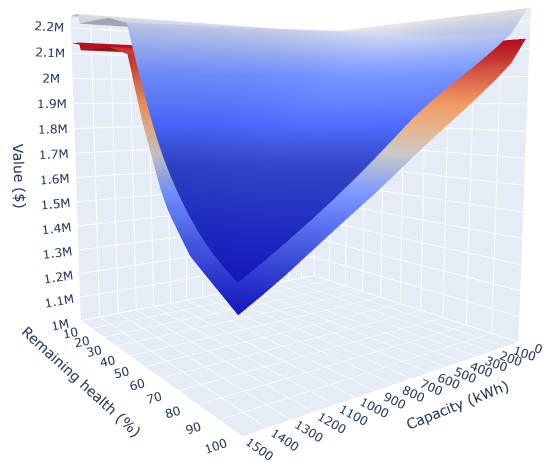}}
  \caption{Resource and price Bellman value functions at day~$1$}
  \label{tts:fig:daily_value_function_1}
\end{figure}
We observe that the resource and price Bellman value functions
present approximately the same shape and are just separated
by a relatively small gap.
The same observation holds true for all days of the horizon.
The largest relative gap between these bounds is 7.90\%.
The relative gap at the inital state (no battery), is around 4.84\%.

We gather in Table~\ref{tts:tab:comptimes_offline} the computing
times of the two decomposition algorithms, namely the total CPU
times and the times when parallelization is on (wall time).
The computation is run on an Intel i7-8850H CPU @ 2.60GHz
6 cores with 16 GB RAM.
The table displays the times needed to compute the intraday
functions and the Bellman value functions.
\begin{table}[hbtp]
  \centering
  \begin{tabular}{l|cc}
                               & ~~~~\textbf{{Price}}~~~~ & \textbf{{Resource}} \\
  \hline
  Intraday functions CPU time  & $1053$ s                 & $2836$ s            \\
  Value functions CPU time     & $6221$ s                 & $1515$ s            \\
  Total CPU time               & $7274$ s                 & $4351$ s            \\
  \hline
  Intraday functions wall time & $267$ s                  & $714$ s             \\
  Value functions wall time    & $2227$ s                 & $1310$ s            \\
  Total wall time              & $2494$ s                 & $2024$ s            \\
  \end{tabular}
  \caption{Computing times of decomposition methods}
\label{tts:tab:comptimes_offline}
\end{table}
We observe that, whereas the price decomposition algorithm requires
a significantly longer CPU time than the resource decomposition
algorithm, the two decomposition algorithms require a
comparable wall time\footnote{Wall time measures how much time
has passed for executing the code, as if you were looking at
the clock on your wall.} when parallelization is on.
The main reason is that the parallelization of the computation of
price Bellman value functions decreases more significantly the computing
time than the parallelization for resource Bellman value functions.
The explanation is that the computation done in parallel is longer
in the price case, hence the CPU time saved is not compensated by
too frequent memory sharings. The price intraday functions
are also faster to compute because the price space is more coarsely
discretized than the exchangeable energy space.

Finally, in Table~\ref{tts:tab:bounds}, we give the values~$\widehat{\overline{V}}_{0}^{\mathrm{R}}(x_{0})$
and~$\widehat{\underline{V}}_{0}^{\mathrm{P}}(x_{0})$
of the resource and price Bellman value functions at day~$d=0$
for the initial state $x_{0}=(s_{0},h_{0},c_{0})=(0,0,0)$, that is,
no battery present. According to Sect.~\ref{tts:sec:bounds},
these values are respectively an upper bound and a lower bound
of the optimal value of Problem~\eqref{eq:2tsmotivlong}.
Note that the numerically computed bounds may fail to be upper
and lower bounds since the resource and price intraday functions
are obtained
\begin{itemize}
\item using some approximation as explained
in~\S\ref{ssect:simpl-ressource} and~\S\ref{ssect:simpl-price},
\item assuming that the noises random variables are independent
at the half hourly time step.
\end{itemize}
\begin{table}[hbtp]
\centering
\begin{tabular}{l|cc}
                         & ~~~~\textbf{{Price}}~~~~ & \textbf{{Resource}} \\
\hline
Lower (price) and upper (resource) bounds  & $2.14$ M\$ & $2.24$ M\$
\end{tabular}
\caption{Bounds obtained by resource and price decomposition}
\label{tts:tab:bounds}
\end{table}

\subsubsection{Simulation of the resource and price policies}
\label{sect:simulating_policies}

We present now several simulation results.
Table~\ref{tts:tab:comptimes_online} displays the times needed
to perform a 20 years simulation over one scenario, from which
we deduce the average time needed to compute a decision at each
time step.
\begin{table}[hbtp]
  \centering
  \begin{tabular}{l|cc}
                               & ~~~~\textbf{{Price}}~~~~ & \textbf{{Resource}} \\
  \hline
  \hline
  Average time to simulate a scenario & $6.19$ s      & $5.22$ s                \\
  Average time to compute a decision  & $17.7$ $\mu$s & $14.9$ $\mu$s           \\
  \end{tabular}
  \caption{Computing times of simulation}
\label{tts:tab:comptimes_online}
\end{table}

\subsubsubsection{Simulation using white noise scenarios}

We draw~$100$ scenarios of battery prices and netloads over
20~years, by assuming that the noise random variables are
independent at the half hourly time step, as detailed in
Remark~\ref{rem:data_independence}.
These simulations are made with the same probabilistic assumptions
as those made to compute the resource and price Bellman value
functions, and therefore make it possible to verify the validity
of the bounds provided by these Bellman functions.
Then, as explained in~\S\ref{sec:computing_policies}, we simulate
the charge and renewal decisions that are made along these scenarios.
All simulations start from an initial state where no battery is present.
The average costs of these scenarios are given in
Table~\ref{tts:tab:cost_white_scenarios}.
\begin{table}[hbtp]
\centering
\begin{tabular}{l|cc}
                         & ~~~~\textbf{{Price}}~~~~ & \textbf{{Resource}} \\
\hline
Average cost over $100$ white noise scenarios & $2.18$ M\$ & $2.25$ M\$\\
\end{tabular}
  \caption{Average simulation costs using white noise scenarios}
\label{tts:tab:cost_white_scenarios}
\end{table}
These two values obtained by simulation being themselves statistical
upper bounds of the true optimal cost, they do not contradict the bound
values given in Table~\ref{tts:tab:bounds}.

\subsubsubsection{Simulation using ``true'' scenarios}

We draw~$1,000$ ``true'' scenarios of battery prices and netloads
over 20~years, that is, scenarios extracted from the realistic
data of the problem. There is thus no more independence assumption
available for these scenarios. Then, we simulate the charge and
renewal decisions that are made when using the intraday functions
and the Bellman value functions obtained by resource and price
decomposition, in order to compare the performances of both methods.
All simulations start from an initial state where no battery is present.
The average costs of these scenarios are given in
Table~\ref{tts:tab:cost_true_scenarios}.
\begin{table}[hbtp]
\centering
\begin{tabular}{l|cc}
                         & ~~~~\textbf{{Price}}~~~~ & \textbf{{Resource}} \\
\hline
Average cost over $1,000$ true scenarios & $2.83$ M\$ & $2.86$ M\$\\
\end{tabular}
  \caption{Average simulation costs using original scenarios}
\label{tts:tab:cost_true_scenarios}
\end{table}
The comparison of the average costs  shows that both decomposition
methods provide comparable performances. However, the price decomposition
outperforms the resource decomposition by achieving on average 1.05\%
of additional economic savings. This slightly superior performance
of the price decomposition is observed on every simulation scenario,
and reveals that the price Bellman value functions might be closer
to the true Bellman value functions.

We also note that the average costs are $20$\% to $25$\% higher
than the corresponding values of the Bellman value functions at the
initial day for the initial state given in Table~\ref{tts:tab:bounds}.
This is a somewhat surprising result since  the Bellman value
obtained by resource (resp. price) decomposition
is an upper bound (resp. lower ) of the optimal value of the problem.
One would therefore expect the average values obtained in simulation
to be between these bounds, which is not the case.
The contradiction comes from the fact that the bounds are computed
assuming that the prices of the batteries are independent day by day,
whereas the simulations are made with scenarios where the prices are
likely to be strongly correlated.

\subsubsubsection{Analysis of some scenarios}

We select three scenarios (the colored scenarios
in Figure~\ref{tts:fig:batcosts}) among the~$1,000$
scenarios of battery prices and netload over 20~years
that we used in the previous paragraph, and we analyse
the behavior of the policies induced by resource and
price decompositions. We recall that all simulations
start from an initial state where no battery is present.
Figure~\ref{tts:fig:health_simulations_bis} displays
the health (or exchangeable energy in kWh) of the batteries
at the end of each day for the three scenarios,
\begin{figure}[hbtp]
  \centering
  \includegraphics[width=0.9\linewidth]{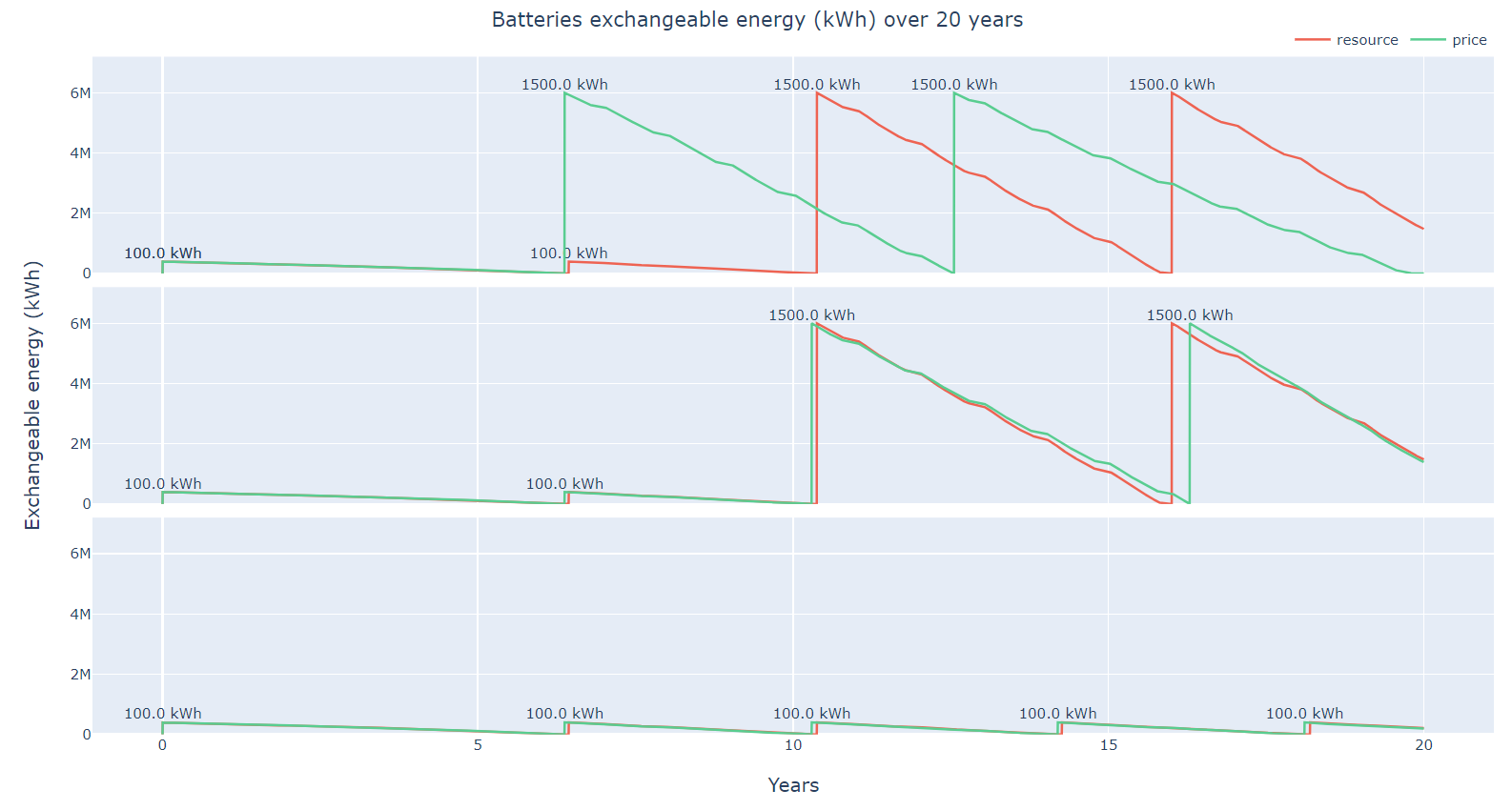}
  \caption{Three simulations of the evolution of the battery
           health over 20 years}
  \label{tts:fig:health_simulations_bis}
\end{figure}
and Table~\ref{tts:tab:simresults} gives the associated simulation costs.
In the first simulation, price and resource decompositions
lead to significantly different renewal decisions and different costs.
A small battery, ($100$ kWh, that is, $400$ kWh of exchangeable
energy\footnote{The integer function~$\fhealth$ in
Equation~\eqref{tts:dynamics-detailed-last-H} is such
that~$\fhealth(100) = \fhealth(1500) = 4$.})
is purchased at day~$d=0$ for both price and resource
decomposition. But at day~$d=2328$, another small battery
is purchased in resource decomposition, whereas a large battery
($1,500$ kWh, that is, $6$ MWh of exchangeable energy) is purchased
in price decomposition. Then, over the remaining time horizon, there
is one battery renewal in price decomposition, and two renewals
in resource decomposition.
In the second simulation, resource and price decompositions produce
very similar health trajectories and costs. This is even clearer for
the third simulation for which the health trajectories and the costs
are almost identical. The third simulation shows a case where battery
prices are high, hence only small batteries, that is, $100$ kWh, are
purchased.
\begin{table}[hbtp]
  \centering
  \begin{tabular}{l|ccc}
    & \textbf{{Scenario 1}} & \textbf{{Scenario 2}} & \textbf{{Scenario 3}} \\
    \hline
    Total cost (resource)
    & $2.757$ M\$           & $2.825$ M\$    &    $3.200$ M\$  \\
    Total cost (price)
    & $2.722$ M\$           & $2.820$ M\$    &    $3.199$ M\$   \\
  \end{tabular}
  \caption{Simulation results along three scenarios}
\label{tts:tab:simresults}
\end{table}
Price decomposition outperforms resource decomposition on the three
scenarios, but only by $1.27$ \% on Scenario~$1$ while the renewal
decisions are significantly different. Our interpretation is that,
in Scenario~$1$, it is almost as rewarding to buy either a big battery
or a small battery taking into account the investment. Moreover,
it seems that resource decomposition slightly underestimates
the benefits of using a large battery compared to a small one.
Indeed, we present in Figure~\ref{fig:tts:value_function_2328}
the resource and price Bellman value functions of the day~$d=2328$
(first battery renewal in Scenario~$1$), when the health
of the battery is fixed to the value~$\overline{H}$ associated
with a large battery renewal (1,500 kWh), that is,
$\widehat{\overline{V}}_{2328}^{\mathrm{R}}(\overline{H},\cdot)$
in~\eqref{eq:fully-simplified-resource-bellman-equation} and
$\widehat{\underline{V}}_{2328}^{\mathrm{P}}(\overline{H},\cdot)$
in~\eqref{eq:fully-simplified-price-bellman-equation}.
We observe that the resource Bellman value function is
significantly higher than the price Bellman value function.

\begin{figure}[hbtp]
  \centering
  \includegraphics[width=0.8\linewidth]{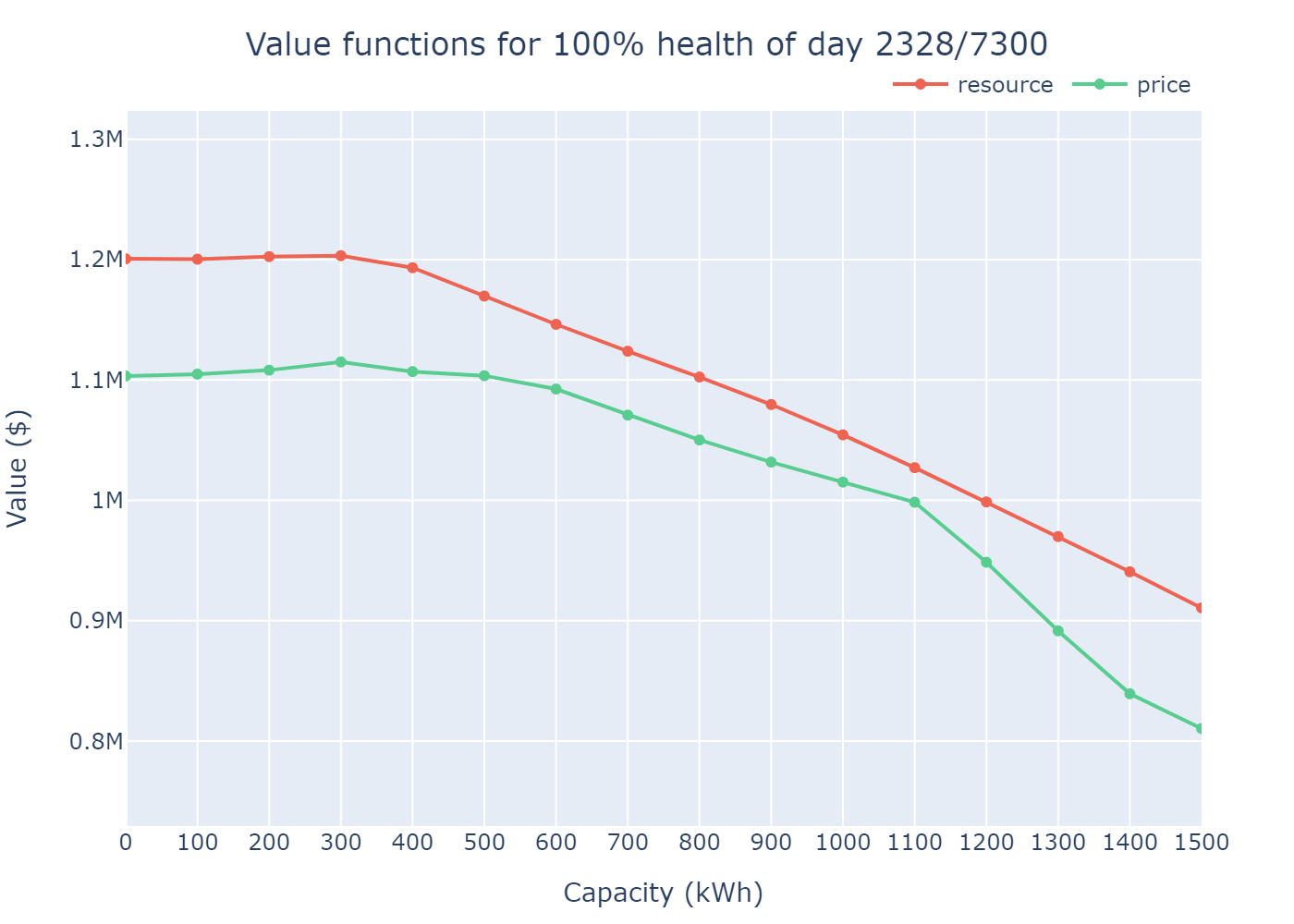}
  \caption{Bellman value function of first renewal day of Scenario~1
           for health fixed at $100$ \%}
\label{fig:tts:value_function_2328}
\end{figure}

\section{Conclusion}

We introduced the formal definition of a two-time-scale stochastic
optimization problem. The motivation for two-time-scale modeling
originated from a battery management problem over a long term horizon
($20$ years) with decisions being made every $30$~minutes
(charge/discharge). We presented two algorithmic methods to compute
daily Bellman value functions to solve these generic problems --- with
an important number of time steps and decisions on different time
scales --- when they display monotonicity properties. Both methods rely
on a Bellman equation applied at the slow time scale, producing
Bellman value functions at this scale.

Our first method, called resource decomposition algorithm, is a primal
decomposition of the daily Bellman equation that requires to compute
the value of a multistage stochastic optimization problem parameterized
by a stochastic resource. The monotonicity property here makes it possible
to relax the coupling constraint and to replace the stochastic resource by
a deterministic one, yielding an upper bound for the daily Bellman value
functions. Instead of this simplification, we could have turned the almost
sure coupling constraint into a constraint in expectation. It would be
interesting to compare this with our approach.

We address a similar and related difficulty in our price decomposition
algorithm. It requires the computation of the value of a stochastic
optimization problem parameterized by a stochastic price. Once again
we replace it by a deterministic price, which is equivalent to dualize
an expectation target constraint.
This makes the previous enhancement proposal even more relevant. Still
our algorithm produces a lower bound that reveals itself tight
in our numerical experiment (we already have observed this numerical
favorable phenomenon in \cite{Carpentier-Chancelier-DeLara-Pacaud:2020}).

Finally we proved with a realistic numerical application that
these methods make it possible to compute design and control policies
for problems with a very large number of timesteps.
But they could also be used for single timescale problems
that exhibit monotonicity, periodicity and a large number of time steps.

\appendix

\section{Complexity of the decomposition algorithms}
\label{tts:ann-complexity}

We compute the complexity of the resource and price decomposition
algorithms~\ref{tts:alg:2tssdp} and~\ref{tts:alg:2tssdpdual}
in terms of number of operations required to implement them,
and we compare them to a brute force use of \DP.

We denote by~$N_{x}^{\mathrm{s}}$ (resp.~$N_{u}^{\mathrm{s}}$
and~$N_{w}^{\mathrm{s}}$) the dimension of the space of state
(resp. control and noise) variables that change only
at the slow time scale.
We denote by~$N_{x}^{\mathrm{sf}}$ (resp.~$N_{x}^{\mathrm{ff}}$)
the dimension of the space of state variables that change at the
fast time scale with an influence at the slow time scale (resp.
without influence at the slow time scale), and by~$N_{u}^{\mathrm{f}}$
(resp.~$N_{w}^{\mathrm{f}}$) the dimension of the space of control
(resp. noise) noise variables that change at the fast time scale.
In numerical applications, we stick to the battery problem under
consideration in the paper:
\begin{itemize}
\item $N_{x}^{\mathrm{s}}=1$ (capacity $\va{c}_{d}$),
  $N_{x}^{\mathrm{sf}}=1$ (health $\va{h}_{d,m}$),
  $N_{x}^{\mathrm{ff}}=1$ (state of charge $\va{s}_{d,m}$),
\item $N_{u}^{\mathrm{s}}=1$ (renewal $\va{r}_{d}$),
  $N_{u}^{\mathrm{f}}=1$ (charge/discharge $\va{u}_{d,m}$)
\item $N_{w}^{\mathrm{s}}=1$ (price $\va{p}_{d}^{b}$),
  $N_{w}^{\mathrm{f}}=1$ (net demand $\va{d}_{d,m}$).
\end{itemize}
As detailed in Remark~\ref{rem:data_independence}, we
assume here stagewise independence of noises at the fast
time scale, so that all intraday problems can be solved by \DP.
To make things simple, we assume that each one-dimensional variable
is discretized in~$\discone$ values, and that each elementary
minimization is conducted by exhaustive search in the control space.
We finally assume that the number~$D+1$ of slow time steps and
the number~$M+1$ of fast time steps in a slow time step are rather
large, whereas the number~$I$ of periodicity classes is rather small.

We first compute the complexity of solving
Problem~\eqref{eq:2tsmotivlong} by \DP. Taking into account
that the dimension of the state at each time step
is~$N_{x}^{\mathrm{ff}}+N_{x}^{\mathrm{sf}}+N_{x}^{\mathrm{s}}$
and that the last fictitious time step at the fast time scale only
involves the control~$\va{r}_{d}$ and the  noise~$\va{p}_{d}^{b}$,
whereas the~$M+1$ previous steps involve the control
$(\va{u}_{d,m}^{+},\va{u}_{d,m}^{+})$ and the noise~$\va{d}_{d,m}$,
the number of elementary operations required to solve the problem is:
\begin{align}
  \label{tts:complex-DP}
  & (D+1)
    \Bp{ \discone^{N_{x}^{\mathrm{ff}}+N_{x}^{\mathrm{sf}}+N_{x}^{\mathrm{s}}+
    N_{u}^{\mathrm{s}}+N_{w}^{\mathrm{s}}} + (M+1)
    \bp{ \discone^{N_{x}^{\mathrm{ff}}+N_{x}^{\mathrm{sf}}+N_{x}^{\mathrm{s}}+
    N_{u}^{\mathrm{f}}+N_{w}^{\mathrm{f}}} } } \\
  & \approx \; D \bp{ \discone^{5} + M \discone^{5}} \nonumber \\
  & \approx \; D  M \discone^{5} \nonumber \eqfinp
\end{align}

\subsection{Resource decomposition algorithm}
\label{tts:ann-resource}

We now compute the complexity of the resource decomposition
algorithm, which depends on the complexity of the intraday
problem numerical solving, and on the way the Bellman equation is solved.

\begin{itemize}
\item We first have to solve the intraday problem, that is,
  we compute the optimal value of the intraday problem for each possible
  value of the initial state~$(s_{d},h_{d},c_{d})$ and each possible
  value of the final resource target~$(s_{d+1},h_{d+1},c_{d+1})$. We
  apply the simplifications introduced at~\S\ref{ssect:simpl-ressource}.
  \begin{itemize}
  \item A first simplification arises when separating the slow and
    fast time scales, that is, the slow components of the problem (state,
    control and noise) are no more taken into account in the computation
    of the optimal value of the intraday problem, which leads to
    Problem~\eqref{tts:eq:intrapbrelaxed-renewal-modified}.
  \item A second simplification is based on the assumption that
    the fast component of the state does not influence the slow
    dynamics and thus can be set to zero at the beginning and
    at the end of each slow time step.
  \item The last simplification is that the modeling of the problem
    is such that it only depends on the health difference during
    a slow time step.
  \end{itemize}

  The resulting intraday problem
  $\widehat{\coutint}_{d}^{\mathrm{R}}(\Delta h_{d},c_{d})
   \defegal \coutint_{d}^{\mathrm{R}}(0,\Delta h_{d},c_{d},0,0)$
  in Equation~\eqref{eq:LR-depends-on-health-delta-on-no-s}
  is solved by \DP\ involving all fast components of the problem,
  and has to be computed for all possible values of the capacity~$c_{d}$
  (the health difference~$\Delta h_{d}$ is part of the initial state
  of the problem and thus the associated complexity is taken into
  account by \DP). The complexity of computing the value
  $\widehat{\coutint}_{d}^{\mathrm{R}}$ for a given value of~$c_{d}$
  is $(M+1) \discone^{N_{x}^{\mathrm{ff}}+N_{x}^{\mathrm{sf}}+
       N_{u}^{\mathrm{f}}+N_{w}^{\mathrm{f}}}$, and the whole
  complexity of computing all values of all intraday functions is
  \begin{equation}
    \label{tts:complex-intraday-resource}
    I (M+1) \discone^{N_{x}^{\mathrm{s}}}
    \discone^{N_{x}^{\mathrm{ff}}+N_{x}^{\mathrm{sf}}+
      N_{u}^{\mathrm{f}}+N_{w}^{\mathrm{f}}}
    \; \approx \; I M \discone^{5} \eqfinv
  \end{equation}
  where we recall that~$I$ is the number of classes of periodicity
  of the problem.
\item Having at disposal all possible
  functions~$\widehat{\coutint}_{d}^{\mathrm{R}}$, we compute the Bellman
  value functions~$\widehat{\overline{V}}_{d}^{\mathrm{R}}$ given by
  Equation~\eqref{eq:fully-simplified-resource-bellman-equation}.
  The complexity associated with that recursion is
  \begin{equation}
    \label{tts:complex-recursion-resource}
    (D+1) \discone^{N_{x}^{\mathrm{sf}}+N_{x}^{\mathrm{s}}+
      N_{x}^{\mathrm{sf}}+N_{u}^{\mathrm{s}}+N_{w}^{\mathrm{s}}}
    \; \approx \; D \discone^{5} \eqfinp
  \end{equation}
\end{itemize}

We are now able to compare the complexity
\eqref{tts:complex-intraday-resource}--\eqref{tts:complex-recursion-resource}
with the complexity~\eqref{tts:complex-DP} of brute force \DP.
The resource decomposition algorithm is relevant if the following
ratio is small
\begin{equation*}
  \mathfrak{R}^{\mathrm{R}} =
  \frac{I M \discone^{5} + D \discone^{5}}
  {D M \discone^{5}} \ll 1 \eqfinv
\end{equation*}
or equivalently
\begin{equation}
  \label{eq:relevant-resource}
  \mathfrak{R}^{\mathrm{R}} =
  \frac{I}{D} + \frac{1}{M} \ll 1 \eqfinp
\end{equation}
Let us consider the relevance condition~\eqref{eq:relevant-resource}
in three different situations, with a number of periodicity classes
equal to~$4$.
\begin{enumerate}
\item \emph{Horizon: $20$ years, fast time step: $1/2$ hour}, slow time step: $1$ day. \\
  Then, we have~$(D;M)=(7,\!300;48)$ and the ratio in~\eqref{eq:relevant-resource}
  is $\mathfrak{R}^{\mathrm{R}} \approx 1/50$.
\item \emph{Horizon: $20$ years, fast time step: $1/2$ hour}, slow time step: $1$ week. \\
  Then, we have~$(D;M)=(1,\!040;336)$ and the ratio in~\eqref{eq:relevant-resource}
  is $\mathfrak{R}^{\mathrm{R}} \approx 1/150$.
\item \emph{Horizon: $20$ years, fast time step: $1/2$ hour}, slow time step: $1$ month. \\
  Then, we have~$(D;M)=(240;1,\!440)$ and the ratio in~\eqref{eq:relevant-resource}
  is $\mathfrak{R}^{\mathrm{R}} \approx 1/60$.
\end{enumerate}
This illustrates that the ratio~$\mathfrak{R}^{\mathrm{R}}$
is minimal when the two quantities~$D$ and~$IM$ are of the same
order of magnitude,\footnote{The solution~$(D\opt,M\opt)$ of the
  optimization problem~$\min_{D,M} \frac{I}{D} + \frac{1}{M}$
  subject to~$D M = \alpha$ is such that~$D\opt = I M\opt$.}
in which case the condition~$I \ll D$ ensures large computer time savings.

In conclusion, the fact that the resource decomposition algorithm
is such that the intraday functions can be computed offline and
that the number of these functions is much less than the number
of slow time steps makes the method very appealing from the
computer time point of view.

\begin{remark}
  It is interesting to compare this complexity result with the
  one obtained by only separating the slow and fast time scales,
  that is, when implementing the resource decomposition algorithm
  by computing all possible values of the intraday problem
  \eqref{tts:eq:intrapbrelaxed-renewal-modified} and then
  by computing the Bellman value functions using
  \eqref{tts:eq:bellman-relaxed-det-renewal-modified}.
  The computation by \DP\ of
  $\coutint_{d}^{\mathrm{R}}(s_{d},h_{d},c_{d},s_{d+1},h_{d+1})$
  for a given $3$-tuple $(c_{d},s_{d+1},h_{d+1})$ requires
  $(M+1) \discone^{N_{x}^{\mathrm{ff}}+N_{x}^{\mathrm{sf}}+
    N_{u}^{\mathrm{f}}+N_{w}^{\mathrm{f}}}$,
  so that the whole complexity of computing the values of
  all intraday functions is
  \begin{equation*}
    I (M+1) \discone^{N_{x}^{\mathrm{s}}+N_{x}^{\mathrm{ff}}+N_{x}^{\mathrm{sf}}}
    \discone^{N_{x}^{\mathrm{ff}}+N_{x}^{\mathrm{sf}}+
      N_{u}^{\mathrm{f}}+N_{w}^{\mathrm{f}}} \\
    \; \approx \; I M \discone^{7} \eqfinp
  \end{equation*}
  The complexity of the Bellman recursion
  \eqref{tts:eq:bellman-relaxed-det-renewal-modified} is
  \begin{equation*}
    (D+1)
    \discone^{N_{x}^{\mathrm{ff}}+N_{x}^{\mathrm{sf}}+
      N_{x}^{\mathrm{s}}+N_{x}^{\mathrm{ff}}+N_{x}^{\mathrm{sf}}+
      N_{u}^{\mathrm{s}}+N_{w}^{\mathrm{s}}} \\
    \; \approx \; D \discone^{7} \eqfinv
  \end{equation*}
  so that Condition~\eqref{eq:relevant-resource} for the method
  to be interesting becomes
  \begin{equation*}
    \discone^{2} \Bp{\frac{I}{D} + \frac{1}{M}} \ll 1 \eqfinv
  \end{equation*}
  a less favorable condition than the one obtained
  where simplifications are taken into account.
\end{remark}

\subsection{Price decomposition algorithm}
\label{tts:ann-price}

We consider now the computation of the intraday problem arising
from the price decomposition algorithm and the associated Bellman
recursion. Taking into account all  simplifications introduced
at~\S\ref{ssect:simpl-price} and developed at~\S\ref{ssect:simpl-price},
the resulting intraday function
$\widehat{\coutint}_{d}^{\mathrm{P}}(c_{d},p^{h}_{d+1})
 \defegal \coutint_{d}^{\mathrm{P}}(0,c_{d},0,p^{h}_{d+1})$
given by~\eqref{tts:eq:intrapbdual-renewal-modified} is solved
by \DP\ involving all fast components of the problem, and has
to be computed for all possible values of the capacity~$c_{d}$
and of the multiplier~$p^{h}_{d+1}$. The complexity of computing the
value $\widehat{\coutint}_{d}^{\mathrm{P}}(c_{d},p^{h}_{d+1})$
for a given couple $(c_{d},p^{h}_{d+1})$ is
$(M+1) \discone^{N_{x}^{\mathrm{ff}}+N_{u}^{\mathrm{f}}+N_{w}^{\mathrm{f}}}$,
and the whole complexity of computing all values of all intraday
functions is
\begin{equation}
  \label{tts:complex-intraday-price}
  I (M+1) \discone^{N_{x}^{\mathrm{s}}+N_{x}^{\mathrm{sf}}}
  \discone^{N_{x}^{\mathrm{ff}}+N_{u}^{\mathrm{f}}+N_{w}^{\mathrm{f}}}
  \; \approx \; I M \discone^{5} \eqfinp
\end{equation}
Then we compute the Bellman value function
$\widehat{\underline{V}}_{d}^{\mathrm{P}}(h_{d},c_{d})$
given by Equation~\eqref{eq:fully-simplified-price-bellman-equation}.
The complexity associated with the Bellman recursion is
\begin{equation}
  \label{tts:complex-recursion-price}
  (D+1) \discone^{N_{x}^{\mathrm{sf}}+N_{x}^{\mathrm{s}}+N_{x}^{\mathrm{sf}}+
    N_{x}^{\mathrm{sf}}+N_{u}^{\mathrm{s}}+N_{w}^{\mathrm{s}}}
  \; \approx \; D \discone^{6} \eqfinp
\end{equation}
Comparing the whole complexity of the price decomposition algorithm
with the complexity~\eqref{tts:complex-DP} of \DP, we conclude that
the resource decomposition algorithm is relevant if the following
ratio is small
\begin{equation*}
  \mathfrak{R}^{\mathrm{P}} =
  \frac{I M \discone^{5} + D \discone^{6}}
  {D M \discone^{5}} \ll 1 \eqfinv
\end{equation*}
or, equivalently, if
\begin{equation}
  \label{eq:relevant-price}
  \mathfrak{R}^{\mathrm{P}} =
  \frac{I}{D} + \frac{10}{M} \ll 1 \eqfinp
\end{equation}
Compared with the relevance condition~\eqref{eq:relevant-resource}
obtained for the resource decomposition algorithm, we conclude
that the price decomposition algorithm is more demanding than
the resource decomposition algorithm.

\section{Proving monotonicity of a battery management problem}
\label{tts:sec:studybatproblem}

As stated in Proposition~\ref{prop:dailymonotone}, we have to satisfy
Assumption~\ref{tts:hyp:nonincreasing} in order to obtain an upper
bound for Problem~\eqref{eq:2tsmotivlong} when using the resource
decomposition algorithm. As explained in \S\ref{ssect:simpl-ressource},
since the capacity of the battery is handled at the slow time scale,
we only relax the equality constraints corresponding to the dynamics
of the state of charge and the health of the battery, so that
the monotonicity property has to be proven only for these two components
of the state.

This is why we focus on the simplified problem~\eqref{tts:eq:gensocprob},
where the state of the battery is its state of charge~$\va{s}_{t}$
and its health~$\va{h}_{t}$, where the decision variable~$\va{u}_t$
is the charge/discharge of the battery at time~$t$ and where ~$\va{w}_t$
is the random net demand. The objective is to minimize the consumption
of power of the external grid\footnote{We denote by~$x^+ = \max(0,x)$
  (resp. $x^- = -\min(0,x)$) the positive (resp. negative) part of a
  real number~$x$.}, that is,
\begin{subequations}
  \label{tts:eq:gensocprob}
  \begin{align}
    \inf_{\na{\va{u}_{t}}_{t\in\timeset{0}{T-1}}}
    & \EE \bgc{ \sum_{t=0}^{T-1} c_{t} \bp{\va{u}_{t} - \va{w}_{t}}^{+}
      + K(\va{s}_{T},\va{h}_{T})} \eqfinv \\
    \text{s.t.} \quad
    & \va{s}_{0} = s_{0}  \eqsepv \va{h}_{0} = h_{0} \eqfinv \\
    & \va{s}_{t+1} = \va{s}_{t} + \rho^{\mathrm{c}} \va{u}_{t+1}^+
      - \rho^{\mathrm{d}} \va{u}^{-}_{t+1}
      \eqfinv \\
    & \va{h}_{t+1} = \va{h}_{t} - \va{u}^{+}_{t+1} - \va{u}^-_{t+1} \eqfinv\\
    & \underline{S} \leq \va{s}_t \leq \overline{S} \eqsepv
      \va{h}_t \geq 0 \eqsepv
      \underline{U} \leq \va{u}_t \leq \overline{U} \eqfinv \\
    & \sigma(\va U_t) \subset \sigma(\va{w}_{0},\dots,\va W_{t}) \eqfinp
  \end{align}
\end{subequations}
We also assume that the noises~$(\va{w}_0,\dots,\va{w}_T)$ are stagewise
independent. Then Problem~\eqref{tts:eq:gensocprob} can be solved by \DP:
\begin{subequations}
  \label{tts:eq:gensocbell}
  \begin{align}
    & V\opt_{T} = K \eqfinv \\
    & V\opt_t(s,h) = \EE \bc{ V_{t}(s,h,\va{w}_{t})} \eqfinv
  \end{align}
  where
  \begin{align}
    V_{t}(s,h,w) = \inf_{u} \;
    & \Bp{c_{t} (u-w)^{+} + V\opt_{t+1}\bp{s + \rho^{\mathrm{c}} u^{+}
      - \rho^{\mathrm{c}} u^{-}, h - u^{+} - u^{-}}}
      \eqfinv \label{tts:eq:gensocopt} \\
    \text{s.t.} \quad
    & \underline{S} - s \leq \rho^{\mathrm{c}} u^+
      - \rho^{\mathrm{d}} u^- \leq \overline{S} -s \eqsepv
      u^{+} + u^{-} \leq h  \eqsepv
      \underline{U} \leq u \leq \overline{U} \eqfinp \label{tts:eq:gensocoptc}
  \end{align}
\end{subequations}
We assume that the final cost function~$K$ is nonincreasing.

\begin{lemma}
  The Bellman value functions~$\na{V\opt_{t}}_{t=0,\dots,T}$
  are nonincreasing in their arguments.
\end{lemma}

\begin{proof}
  The proof is done by induction.
  The last Bellman value function~$V\opt_{T}$ is nonincreasing by assumption.

  Assume that~$V\opt_{t+1}$ is nonincreasing.
  \begin{itemize}
  \item The Bellman value function $V_{t}$ is nonincreasing in~$h$,
    as decreasing~$h$ reduces the admissible set of Problem~\eqref{tts:eq:gensocopt} and increases the objective
    as~$V\opt_{t+1}$ is nonincreasing.
  \item To show that $V_{t}$ is nonincreasing in~$s$, consider
    two states of charge~$s$ and~$s'$ such that~$s' \geq s$,
    and let~$u\opt(s)$ be an $\epsilon$-optimal solution of
    Problem~\eqref{tts:eq:gensocopt}-\eqref{tts:eq:gensocoptc}:
    \begin{equation*}
      c_t  (u\opt(s)-w)^+ +
      V\opt_{t+1}\bp{s + \rho^{\mathrm{c}} \bp{u\opt(s)}^{+}
        - \rho^{\mathrm{d}} \bp{u\opt(s)}^{-},
        h - \bp{u\opt(s)}^{+} - \bp{u\opt(s)}^{-}}
      \leq V_t(s,h,w)+ \epsilon \eqfinp
    \end{equation*}
    We distinguish two cases.
    \begin{itemize}
    \item $u\opt(s) \leq 0$: then $u\opt(s)$ is admissible for~$V_t(s',h,w)$ because
      \begin{equation*}
        \underline{S} - s'
        \; \leq \; \underline{S} - s
        \; \leq \; \rho^{\mathrm{c}} \bp{u\opt(s)}^{+} - \rho^{\mathrm{d}} \bp{u\opt(s)}^{-}
        \; =    \; \rho^{\mathrm{d}} u\opt(s)
        \; \leq \; 0 \leq \overline{S} -s' \eqfinp
      \end{equation*}
      Moreover, as~$V_{t+1}$ is nonincreasing, we have
      \begin{multline*}
        V\opt_{t+1}\bp{s' + \rho^{\mathrm{c}} \bp{u\opt(s)}^{+}
          - \rho^{\mathrm{d}} \bp{u\opt(s)}^{-},
          h - \bp{u\opt(s)}^{+} - \bp{u\opt(s)}^{-}} \\
        \leq V\opt_{t+1}\bp{s + \rho^{\mathrm{c}} \bp{u\opt(s)}^{+}
          - \rho^{\mathrm{d}} \bp{u\opt(s)}^{-},
          h - \bp{u\opt(s)}^{+} - \bp{u\opt(s)}^{-}} \eqfinv
      \end{multline*}
      so that
      \begin{equation*}
        V_t(s',h,w) \leq V_t(s,h,w)+ \epsilon \eqfinp
      \end{equation*}
    \item $u\opt(s) > 0$:
      let~$u^\flat\np{s'} = \min\ba{{\rho^{\mathrm{c}}}^{-1} (\overline{S} -s'),u\opt(s)}$.
      Then $u^\flat\np{s'}$ is admissible for~$V_t(s',h,w)$ as
      \begin{align*}
        & \underline{U} \leq 0 < u^\flat\np{s'} \leq u\opt(s) \leq \overline{U} \eqfinv \\
        & \underline{S} - s' \leq 0 < \rho^{\mathrm{c}} u^\flat\np{s'}
          \leq \rho^{\mathrm{c}} {\rho^{\mathrm{c}}}^{-1} (\overline{S} -s')
          = \overline{S}-s' \eqfinv \\
        & u^\flat\np{s'} \leq u\opt(s) \leq h \eqfinp
      \end{align*}
      Moreover we have either
      $s' + \rho^{\mathrm{c}} u^\flat\np{s'} = s + \rho^{\mathrm{c}} u\opt(s)$ or
      $s' + \rho^{\mathrm{c}} u^\flat\np{s'}  = \overline{S} \geq s + \rho^{\mathrm{c}} u\opt(s)$,
      so that it always holds true that
      \begin{equation*}
        s' + \rho^{\mathrm{c}} u^\flat\np{s'} \geq s + \rho^{\mathrm{c}} u\opt(s) \eqfinp
      \end{equation*}
      This last inequality and the fact that~$V_{t+1}$ in nonincreasing
      lead to
      \begin{equation*}
        c_t \bp{u^\flat\np{s'}-w}^{+} +
        V\opt_{t+1}\bp{s + \rho^{\mathrm{c}} u^\flat\np{s'}, h - u^\flat\np{s'}}
        \leq c_t \bp{u\opt(s)-w}^{+} +
        V\opt_{t+1}\bp{s + \rho^{\mathrm{c}} u\opt(s) , h - u\opt(s)} \eqfinp
      \end{equation*}
      Finally, we get that
      \begin{equation*}
        V_t(s',h,w) \leq V_t(s,h,w)+ \epsilon \eqfinv
      \end{equation*}
      so that~$V_t$ is nonincreasing is~$s$.
    \end{itemize}
  \end{itemize}
  We conclude that~$V_t$ is nonincreasing in~$(s,h)$.
\end{proof}


\begin{thebibliography}{10}

\bibitem{phdthesisAbgottspon}
H.~Abgottspon.
\newblock {\em Hydro Power planning: Multi-horizon modeling and its
  applications}.
\newblock PhD thesis, ETH Zürich, 08 2015.

\bibitem{abgottsponmedium}
H.~{Abgottspon} and G.~{Andersson}.
\newblock Medium-term optimization of pumped hydro storage with stochastic
  intrastage subproblems.
\newblock In {\em 2014 Power Systems Computation Conference}, pages 1--7, 2014.

\bibitem{abgottspon2016multi}
H.~Abgottspon and G.~Andersson.
\newblock Multi-horizon modeling in hydro power planning.
\newblock {\em Energy Procedia}, 87:2--10, 2016.

\bibitem{tts:bellman57}
R.~Bellman.
\newblock {\em Dynamic Programming}.
\newblock Princeton University Press, New Jersey, 1957.

\bibitem{tts:bertsekas1999nonlinear}
D.~P. Bertsekas.
\newblock {\em Nonlinear programming}.
\newblock Athena Scientific Belmont, 1999.

\bibitem{tts:bertsekas2005dynamic}
D.~P. Bertsekas.
\newblock Dynamic programming and suboptimal control: A survey from {ADP} to
  {MPC}.
\newblock {\em European Journal of Control}, 11(4-5):310--334, 2005.

\bibitem{tts:bertsekas1995dynamic}
D.~P. Bertsekas.
\newblock {\em Dynamic programming and optimal control. {V}ol. {I}}.
\newblock Athena Scientific, Belmont, MA, fourth edition, 2017.

\bibitem{Carpentier-Chancelier-DeLara-Pacaud:2020}
P.~Carpentier, J.-P. Chancelier, M.~{De Lara}, and F.~Pacaud.
\newblock Mixed spatial and temporal decompositions for large-scale multistage
  stochastic optimization problems.
\newblock {\em Journal of Optimization Theory and Applications},
  186(3):985--1005, 2020.

\bibitem{tts:carpentier:JOCA}
P.~Carpentier, J.-P. Chancelier, M.~D. Lara, T.~Martin, and T.~Rigaut.
\newblock {Time Block Decomposition of Multistage Stochastic Optimization
  Problems}.
\newblock {\em Journal of Convex Analysis}, 2023.
\newblock Accepted for publication.

\bibitem{tts:carpentier2017decomp}
P.~Carpentier and G.~Cohen.
\newblock {\em Décomposition-coordination en optimisation déterministe et
  stochastique}.
\newblock Springer-Verlag, Berlin, 2017.

\bibitem{tts:haessig:hal-01147369}
P.~Haessig, H.~B. Ahmed, and B.~Multon.
\newblock Energy storage control with aging limitation.
\newblock In {\em PowerTech, 2015 IEEE Eindhoven}, pages 1--6. IEEE, 2015.

\bibitem{tts:heymann:hal-01349932}
B.~Heymann and P.~Martinon.
\newblock {Optimal battery aging : an adaptive weights dynamic programming
  algorithm}.
\newblock {\em {Journal of Optimization Theory and Applications}},
  179(3):1043--1053, 2018.

\bibitem{KautEA12}
M.~Kaut, K.~T. Midthun, A.~S. Werner, A.~Tomasgard, L.~Hellemo, and M.~Fodstad.
\newblock Multi-horizon stochastic programming.
\newblock {\em Computational Management Science}, 11(1--2):179--193, 2014.
\newblock Special Issue: Computational Techniques in Management Science.

\bibitem{tts:leclere2018exact}
V.~Lecl{\`e}re, P.~Carpentier, J.-P. Chancelier, A.~Lenoir, and F.~Pacaud.
\newblock Exact converging bounds for stochastic dual dynamic programming via
  {Fenchel} duality.
\newblock {\em SIAM Journal on Optimization}, 30(2):1223--1250, 2020.

\bibitem{Maggioni2019BoundsIM}
F.~Maggioni, E.~Allevi, and A.~Tomasgard.
\newblock Bounds in multi-horizon stochastic programs.
\newblock {\em Annals of Operations Research}, pages 1--21, 2019.

\bibitem{Moreau:1970}
J.~J. Moreau.
\newblock Inf-convolution, sous-additivit\'e, convexit\'e des fonctions
  num\'eriques.
\newblock {\em J. Math. Pures Appl. (9)}, 49:109--154, 1970.

\bibitem{tts:Pereira:1991:MSO:3113604.3113829}
M.~V. Pereira and L.~M. Pinto.
\newblock Multi-stage stochastic optimization applied to energy planning.
\newblock {\em Math. Program.}, 52(1-3):359--375, May 1991.

\bibitem{tts:Philpott2016MIDASA}
A.~Philpott, F.~Wahid, and F.~Bonnans.
\newblock {MIDAS: A Mixed Integer Dynamic Approximation Scheme}.
\newblock Research report, {Inria Saclay Ile de France}, June 2016.

\bibitem{tts:powell2007approximate}
W.~B. Powell.
\newblock {\em Approximate Dynamic Programming: Solving the curses of
  dimensionality}, volume 703.
\newblock John Wiley \& Sons, 2007.

\bibitem{pritchard2005hydroelectric}
G.~Pritchard, A.~B. Philpott, and P.~J. Neame.
\newblock Hydroelectric reservoir optimization in a pool market.
\newblock {\em Mathematical programming}, 103(3):445--461, 2005.

\bibitem{tts:rockafellar2015convex}
R.~T. Rockafellar.
\newblock {\em Convex analysis}.
\newblock Princeton university press, 2015.

\bibitem{Rockafellar-Wets:1991}
R.~T. Rockafellar and R.~J.-B. Wets.
\newblock Scenarios and policy aggregation in optimization under uncertainty.
\newblock {\em Mathematics of operations research}, 16(1):119--147, 1991.

\bibitem{skar2016multi}
C.~Skar, G.~Doorman, G.~A. P{\'e}rez-Vald{\'e}s, and A.~Tomasgard.
\newblock A multi-horizon stochastic programming model for the {European} power
  system.
\newblock {\em Subimtted to an international peer reviewed journal, In review},
  2016.

\bibitem{Werner2013}
A.~S. Werner, A.~Pichler, K.~T. Midthun, L.~Hellemo, and A.~Tomasgard.
\newblock {\em Risk Measures in Multi-Horizon Scenario Trees}, pages 177--201.
\newblock Springer US, Boston, MA, 2013.

\bibitem{tts:Zou2018}
J.~Zou, S.~Ahmed, and X.~A. Sun.
\newblock Stochastic dual dynamic integer programming.
\newblock {\em Mathematical Programming}, Mar 2018.

\end{thebibliography}
\end{document}